\documentclass[11pt]{amsart}

\usepackage{graphicx}

\usepackage{appendix}

\usepackage{comment}
\usepackage{enumerate}
\usepackage{esvect}
\usepackage{bm}
\usepackage{amsmath}
\usepackage{amssymb}
\usepackage{mathtools}
\newcommand{\defeq}{\vcentcolon=}
\newcommand{\eqdef}{=\vcentcolon}

\usepackage{mathrsfs}
\usepackage{color}

\usepackage[normalem]{ulem}

\usepackage{url}

\newtheorem{theorem}{Theorem}[section]
\newtheorem{prop}[theorem]{Proposition}
\newtheorem{lemma}[theorem]{Lemma}
\newtheorem{remark}[theorem]{Remark}
\newtheorem{definition}[theorem]{Definition}
\newtheorem{conj}{Conjecture}

\newtheorem{ex}[theorem]{Example}

\usepackage{bbm}

\DeclareSymbolFont{bbold}{U}{bbold}{m}{n}
\DeclareSymbolFontAlphabet{\mathbbold}{bbold}

\usepackage{stmaryrd}

\usepackage{hyperref}
\hypersetup{hidelinks,colorlinks}

\hypersetup{
    linkcolor=red,
    citecolor=blue,
 }

\usepackage{dsfont}

\usepackage{mleftright} % for \mleft and \mright macros

\newcommand{\RR}{\mathbb{R}}
\newcommand{\NN}{\mathbb{N}}

\newcommand{\ZZ}{\mathbb{Z}}

\newcommand{\settingstar}{{\normalfont(\hyperref[setting_star]{$\ast$})}}

\newcommand{\hatted}[1]{\widehat{#1}}

\DeclareMathOperator{\oddleq}{\mathcal{O}_{\leq}}
\DeclareMathOperator{\evenleq}{\mathcal{E}_{\leq}}

\title{Enumerating binary words restricted by subsequence frequency}
\author{Glenn Bruda}

\allowdisplaybreaks

\begin{document}

\begin{abstract}
    Let $p$ be a binary word of length $\ell$ with $r\geq2$ runs. Previously known only for $k\leq4$, we show for $n$ sufficiently large that the number of binary words of length $n$ with exactly $k$ subsequences equal to $p$ is polynomial in $n$ of degree at most $\ell-r+1$ for any positive integer $k$. We also prove a sharp upper bound on the number of subsequences equal to $p$ of a binary word $w$ in terms of the runs of $p$ and $w$.
\end{abstract}

\maketitle

\thispagestyle{empty}

\section{Introduction and Statement of Results}

The combinatorics of subsequences of words has been studied extensively and with great variety. Much of this study has been of \emph{factors} (or substrings), which are necessarily consecutive. The Goulden-Jackson cluster method was introduced in \cite{GJ1,GJ2} to find the generating function for the number of words avoiding a fixed set of words as factors. In \cite{noonan_zeilberger}, Noonan and Zeilberger gave a \textsc{Maple} implementation of the Goulden-Jackson cluster method, and in particular gave an extension to find the generating function for the number of words with exactly a fixed number of factors equal to a word in this fixed set. This generating function was further refined by Bassino, Cl\'ement, and Nicod\`eme in \cite{MGF_GJCM} by also keeping track of the length. Several variations of the frequency of factors have been studied, as Arnoux \cite{circular_words} did with circular words and Christodoulakis, Christou, Crochemore, and Iliopoulos \cite{overlapping} did when the factors are overlapping.

In this article, rather than factors, we consider general subsequences (not necessarily consecutive). Menon and Singh \cite{subseq_freq} studied the number of binary words of length $n$ with exactly $k$ subsequences equal to a fixed binary word $p$, computing this enumeration explicitly for all $0\leq k\leq4$. Our main result, Theorem~\ref{polynomial_theorem}, continues this work of Menon and Singh, describing the form of the number of binary words of length $n$ with exactly $k$ subsequences equal to a fixed binary word for \emph{any} positive integer $k$. As demonstrated by Tian in \cite[Equation 2]{other_formula_for_occurrence_count} and the explicit formulas in \cite{subseq_freq}, this enumeration is increasingly complicated as $k$ becomes larger; thus, a description for all $k$ is quite valuable.

An upper bound on the number of distinct factors of a binary word was given by Shallit in \cite{factor_occurrence_upper_bound}, and maximal subsequence occurrence was studied by Fang in \cite{fang2025maximalnumbersubwordoccurrences}. Recently, the densities of subwords of binary words were studied by Kenyon \cite{kenyon}. Relating to these works, our other result, Theorem~\ref{upper_bound_thm}, gives an upper bound on the number of subsequences equal to $p$ in $w$ for fixed binary words $p$ and $w$, which obtains equality when $p$ and $w$ are alternating. As discussed extensively later, Theorem~\ref{upper_bound_thm} opens a new avenue toward obtaining increasingly better upper bounds on subsequence occurrence in terms of the runs of $p$ and $w$.

The remainder of this section is dedicated to stating Theorems~\ref{polynomial_theorem} and~\ref{upper_bound_thm} and the rest of the preliminaries. Sections~\ref{section_2} and~\ref{section_3} give proof of Theorems~\ref{polynomial_theorem} and~\ref{upper_bound_thm} respectively, and Section~\ref{section_4} discusses potential improvements and generalizations of our results.

A \emph{binary word} is a string $w\in\{0,1\}^n$ whose elements we call \emph{letters}, where we call $n$ the \emph{length} of $w$. Given a binary word $w$ of length $n$, for each $1\leq i\leq n$, we define $w_i$ to be the $i$\textsuperscript{th} letter of $w$. We use this subscript notation throughout this article, including in the following definition of an occurrence, our main object of study.

\begin{definition}[{\cite[pg.2]{subseq_freq}}]
    Let $p=p_1\cdots p_{\ell}$ and $w=w_1\cdots w_n$ be binary words. Then an \textbf{occurrence} of $p$ in $w$ is a choice of indices $1\leq i_1<i_2<\cdots<i_{\ell}\leq n$ such that $w_{i_1}\cdots w_{i_{\ell}}=p$. In other words, an occurrence of $p$ in $w$ is one of the ${n\choose\ell}$ subsequences of $w$ that match $p$.
\end{definition}

The occurrences of $p$ in $w$ may be ordered \emph{lexicographically}: $1\leq i_1<i_2<\cdots<i_{\ell}\leq n$ is lexicographically less than $1\leq j_1<j_2<\cdots<j_{\ell}\leq n$ if there exists a $1\leq b\leq \ell$ such that $i_b<j_b$ and $i_a=j_a$ for all $a<b$. Throughout this article, we call the lexicographically least (resp., greatest) occurrence the first (resp., last) occurrence. We now introduce our two main pieces of notation.

\begin{definition}[{\cite[pg.2]{subseq_freq}}]
    Let $p$ and $w$ be binary words. Let $c_p(w)$ be the number of occurrences of $p$ in $w$. For integers $n,k\geq0$, let $B_p(n,k)=|\{w\in\{0,1\}^n:c_p(w)=k\}|$ be the number of binary words of length $n$ with exactly $k$ occurrences of $p$.
\end{definition}

Continuing the perspective given by Menon and Singh in \cite{subseq_freq}, we find studying the \emph{runs} of binary words to be the key approach toward obtaining our results.

\begin{definition}
    A \textbf{run} of size $L$ in a binary word $w$ is a maximal consecutive subsequence of $w$ of $L$ equal letters.
\end{definition}

For example, the runs of $1110010000100011$ are $111,00,1,0000,1,000,11$. We will find it notationally convenient to write the runs of binary words using \emph{superscript notation}: given a binary word $w$, we write $x^L$ to denote a run of $x$'s of length $L$ in $w$. For example, we write $00011000010$ as $0^31^20^41^10^1$.

\begin{comment}

\begin{definition}
    Let $w$ be a binary word. A \textbf{slot} of $w$ is a space between adjacent letters in $w$, including the spaces before the first letter and after the last letter.
\end{definition}

\begin{definition}
    Two distinct letters in $w$ of equal value $x$ are said to be \textbf{neighboring} if they are separated only by a run of $(1-x)$'s.
\end{definition}

\end{comment}

By \cite[Propositions 3.2, 3.4, 3.8, 3.11]{subseq_freq}, we observe that when $1\leq k\leq 4$ and $p$ has length $\ell$ and $r\geq2$ runs, $B_p(n,k)$ is a polynomial in $n$ of degree at most $\ell-r+1$ for $n$ sufficiently large. From this observation, we ask if this holds for all positive integers $k$. Our main result is the following, which resolves this question in the affirmative.

\begin{theorem}\label{polynomial_theorem}
    Let $k\in\ZZ^+$ and $p$ be a binary word of length $\ell$ with $r\geq2$ runs. Then for $n$ sufficiently large, $B_p(n,k)$ is a polynomial in $n$ of degree at most $\ell-r+1$. Furthermore, if either
    \begin{enumerate}[(i)]
        \item $p$ has a run of length $1$, or
        \item $p$ has a run of length $b\geq2$ at its boundary\footnote{That is, the run is the first or last run of $p$.} such that ${b+s\choose b}=k$ for some $s\geq0$,
    \end{enumerate}
    then $\deg (B_p(n,k))=\ell-r+1$.
\end{theorem}

We note that for $r=1$, the degree of $B_p(n,k)$ can become arbitrarily large with $k$ \cite{noah_communication}. Indeed, letting $p=x^{\ell}$,
\begin{align*}
    B_{p}(n,k)=\begin{cases}
        {n\choose g} & \text{if~there exists~} g \text{~such that~} {g\choose\ell}=k,\\
        0 & \text{otherwise.}
    \end{cases}
\end{align*}
We discuss generalizing Theorem~\ref{polynomial_theorem} to $m$-ary words in Subsection~\ref{polynomial_m_ary_section} and completely characterizing when $\deg (B_p(n,k))=\ell-r+1$ in Subsection~\ref{second_part_section}.

\begin{definition}
    For binary words $p$ and $w$, we say that letter $w_s$ of $w$ is \textbf{used in} an occurrence of $p$ in $w$ if there is an occurrence $\mathbf{i}$ of $p$ in $w$ with an index $i_j$ such that ${i_j}=s$. A run $y^v$ of $w$ is said to be used in an occurrence of $p$ if a letter of $y^v$ is used in an occurrence of $p$.
\end{definition}

Moving toward stating Theorem~\ref{upper_bound_thm}, let $p=p_1\cdots p_{\ell}$ and $w$ be binary words. Note that the first letter of the first occurrence of $p$ in $w$ uses the first instance of $p_1$ in $w$. So if the first run of $w$ is a run of $(1-p_1)$'s, then this run is not used in any occurrence of $p$ in $w$. By symmetry, we also see that if the last run of $w$ is a run of $(1-p_{\ell})$'s, then this run is not used in any occurrence of $p$ in $w$. Thus, the number of occurrences of $p$ in $w$ is invariant under removing these \emph{negligible boundary runs}, if they exist.

\begin{definition}
    Let $p=p_1\cdots p_{\ell}$ be a binary word. A \textbf{negligible boundary run} (NBR) with respect to $p$ of a binary word $w=w_1\cdots w_n$ is either
    \begin{enumerate}[(i)]
        \item a run at the beginning of $w$ such that $w_1\neq p_1$, or
        \item a run at the end of $w$ such that $w_n\neq p_{\ell}$.
    \end{enumerate}
\end{definition}

For example, letting $p=101$ and $w=00110111$, we see that $00$ is an NBR since it does not match the value of $p_1=1$, but $111$ is not an NBR since it does match the value of $p_{\ell}=1$. Removing NBRs is essentially a free improvement for our upper bound on $c_p(w)$, allowing us to instead regard a word potentially shorter than $w$. 

We prove our upper bound on $c_p(w)$ by analysis of an exact combinatorial formula for $c_p(w)$ (Proposition~\ref{combinatorial_formula_for_number_of_occurrences}). Our bound obtains equality when all the runs of $p$ and $w$ are of length $1$; we call such a binary word \emph{alternating}.

\begin{definition}
    For $x\in\RR$, define $\oddleq(x)$ to be the largest odd integer less than or equal to $x$ and $\evenleq(x)$ to be the largest even integer less than or equal to $x$.
\end{definition}

\begin{theorem}\label{upper_bound_thm}
    Let $p$ be a binary word with $r$ runs, minimum run length $r_{\min}$, and maximum run length $r_{\max}$. Let $w$ be a binary word and $w'$ be $w$ with all NBRs removed. Let $R'$ and $R_{\max}'$ be the number of runs and the maximum run length of $w'$, respectively. For each $1\leq i\leq r$, set
    \begin{align*}
        \sigma_{r,R'}(i)=
        \begin{cases}
            \oddleq\mleft(R'/r\mright)+2 & \text{\normalfont if~}i\leq \frac{1}{2}{\evenleq\mleft(R'-r\oddleq\mleft(R'/r\mright)\mright),}\\
            \oddleq\mleft(R'/r\mright) & \text{\normalfont otherwise.}
        \end{cases}
    \end{align*}
    Then
    \begin{align*}
        c_p(w)\leq{\left\lfloor\frac{R'-r}{2}\right\rfloor+r\choose r}\prod_{i=1}^{r}R_{\max}'\left(\frac{R_{\max}'}{2}\left(\sigma_{r,R'}(i)+1\right)-r_{\min}+1\right)^{r_{\max}-1}
    \end{align*}
    with equality when $p$ and $w$ are alternating.
\end{theorem}

\begin{remark}\label{relaxed_upper_bound}
    A simpler version of Theorem~\ref{upper_bound_thm}, which is only slightly weaker (and maintains sharpness), may be obtained by replacing each $\sigma_{r,R'}(i)$ with $R'/r$. That is, using the same notation as in Theorem~\ref{upper_bound_thm},
    \begin{align*}
        c_p(w)\leq{\left\lfloor\frac{R'-r}{2}\right\rfloor+r\choose r}\left(R_{\max}'\left(\frac{R_{\max}'}{2}\left(\frac{R'}{r}+1\right)-r_{\min}+1\right)^{r_{\max}-1}\right)^r,
    \end{align*}
    with equality when $p$ and $w$ are alternating.
\end{remark}

We present this weakening of Theorem~\ref{upper_bound_thm} principally to make the bound's asymptotics conspicuous and simplify computations thereof. Indeed, unless $r$ and $R'$ are fixed, this weakening differs only by a constant factor. However, this constant can be large (see Example~\ref{explicit_computation_bound}), justifying the trouble of regarding the complicated $\sigma_{r,R'}(i)$ sequence.

Letting the lengths of $p$ and $w$ be $\ell$ and $n$ respectively, we remark that we have the \emph{length bound} $c_p(w)\leq{n\choose\ell}$, which follows immediately by the definition of an occurrence of $p$ in $w$. With this length bound noted, we naturally query how Theorem~\ref{upper_bound_thm} compares.

\begin{ex}\label{explicit_computation_bound}
    \normalfont Let $w$ be a binary word with $105$ runs, each of length $20$. Let $p$ be the binary word with $51$ runs of length $2$ beginning and ending with the same letters as $w$. Then the length bound gives $c_p(w)\leq {2100\choose 102}\approx6.3\cdot10^{175}$, whereas Theorem~\ref{upper_bound_thm} gives $c_p(w)\leq {78\choose27}780^{27}380^{24}\approx6.7\cdot10^{160}$. The true value of $c_p(w)$, computed by Sage \cite{sage}, is approximately $9.4\cdot10^{147}$. 
    
    Toward demonstrating how large the constant factor difference between Theorem~\ref{upper_bound_thm} and Remark~\ref{relaxed_upper_bound} may be, we see that for this case, the latter gives $c_p(w)\leq {78\choose 27}(10060/17)^{51}\approx 1.6\cdot10^{162}$, nearly $24$ times larger than as given by the former.
\end{ex}

Fixing the parameters used in our bound defines a family of pairs of words $p$ and $w$ over which Theorem~\ref{upper_bound_thm} affixes a single upper bound on $c_p(w)$. Since our bound depends only on the number of runs of $p$ and $w$ and their extremal lengths, we see that our bound is better when the runs of $p$ and $w$ have about the same lengths (so Example~\ref{explicit_computation_bound} is quintessential for our bound's optimality). Formally, our bound improves as the variances of the distributions of the run lengths of $p$ and $w$ decrease. This is shown asymptotically in Example~\ref{upper_bound_asymptotics}, where we fix $p$ for simplicity.
\begin{ex}\label{upper_bound_asymptotics}
    \normalfont{Let $p$ be a fixed alternating binary word of length $\ell$ and $w$ be of length $n$. For constants $a$ and $b$, let $R=a\sqrt{n}(1+o(1))$ and $R_{\max}=b\sqrt{n}(1+o(1))$ be the number of runs and the maximum run length of $w$, respectively. Then, letting $T(n)$ denote the length bound and $B(n)$ denote our bound (the strong and weak versions are equivalent in this case since $r_{\max}=1$), we have $(ab/2)^{-\ell}B(n)\sim T(n)$. This may be seen by a routine computation done in Appendix~\ref{appendix}. Thus, if the distribution of the run lengths of $w$ has minimal variance (implying that $ab=1$), then our bound is asymptotically better than the length bound by a factor of $2^{\ell}$.}
\end{ex}

Unfortunately, our bound is sometimes worse than the length bound, particularly when $r_{\max}$ and $R_{\max}$ are outliers in the distributions of the run lengths of $p$ and $w$, respectively. However, we remark that Theorem~\ref{upper_bound_thm} is far from optimized, and give several suggested directions for improvement in Subsection~\ref{improving_upper_bound}. We also expect an $m$-ary generalization of Theorem~\ref{upper_bound_thm}, which we discuss in Subsection~\ref{m_ary_bound}, to improve as $m$ becomes large.

\section{Proof of Theorem~\ref{polynomial_theorem}}\label{section_2}

Given a binary word $w$ with a fixed number of occurrences of $p$, one of the key observations in proving the first part of Theorem~\ref{polynomial_theorem} is that if a run of $w$ is too long, then that run cannot be used in an occurrence of $p$ (otherwise, there would be too many occurrences of $p$). The following definition precisely describes what it means for a run to be too long in this context, which we later substantiate to be the correct characterization.

\begin{definition}
    Let $k\in\ZZ^+$ and $p$ be a binary word. Set $r_{\max}$ to be the length of the longest run in $p$. A \textbf{short run} is a run of length less than $k+r_{\max}$, and a \textbf{long run} is a run of length at least $k+r_{\max}$.
\end{definition}

Since we are frequently discussing the enumeration of binary words of length $n$ with exactly $k$ occurrences of $p$ as denoted by $B_p(n,k)$, it is convenient to provide notation for the set of such binary words.

\begin{definition}
    Let $p$ be a binary word. We define $\mathcal{B}_{p}(n,k)$ to be the set of binary words of length $n$ with exactly $k$ occurrences of $p$, wherein we have $|\mathcal{B}_{p}(n,k)|=B_p(n,k)$.
\end{definition}

The cornerstone of the proof of the first part of Theorem~\ref{polynomial_theorem} lies in reducing every word in $\mathcal{B}_p(n,k)$ to a shorter word that retains the occurrence structure. These shorter words, which we call ``skeletons'' of all words with exactly $k$ occurrences of $p$, are much easier to handle. %From a high level, the proof of the first part of Theorem~\ref{polynomial_theorem} can be described as making a series of deductions on these skeletons, and then carefully translating the occurrence structure to the corresponding words in $\mathcal{B}_p(n,k)$, yielding an enumeration of $\mathcal{B}_p(n,k)$. 
The precise description of these skeletons is given in the following definition.

\begin{definition}\label{SW_def}
    Let $p$ be a binary word and $k\in\ZZ^+$. The set of \textbf{skeleton words} $\text{\normalfont SW}_p(k)$ is the set of words on $\{0,1,\hatted{0},\hatted{1}\}$ with exactly $k$ occurrences of $p$ such that
    \begin{enumerate}[(i)]
        \item\label{property_1} all runs of $0$'s and $1$'s are short,
        \item\label{property_2} all runs of $\hatted{0}$'s are of length $1$ and have only $1$'s as surrounding letters,
        \item\label{property_3} all runs of $\hatted{1}$'s are of length $1$ and have only $0$'s as surrounding letters, and
        \item\label{property_4} for $x\in\{0,1\}$, replacing $\hatted{x}$ with a run of $x$'s of arbitrary length does not cause the number of occurrences of $p$ to exceed $k$.
    \end{enumerate}

   % set of ternary words on the alphabet $\{0,1,m\}$ with exactly $k$ occurrences of $p$ such that all runs of $0$'s and $1$'s are short and all runs of $m$'s are of length at most $2\ell-1$. 
\end{definition}

To show that $B_p(n,k)$ is a polynomial of degree at most $\ell-r+1$, we prove that every word in $\mathcal{B}_p(n,k)$ may be obtained by replacing the $\hatted{0}$'s and $\hatted{1}$'s in a word in $\text{SW}_p(k)$ with long runs of $0$'s and $1$'s respectively. Subsequently, given a $w\in\text{SW}_p(k)$, we show that the number of words in $\mathcal{B}_p(n,k)$ that may be obtained by replacing the hatted letters in $w$ with the appropriate long runs is a polynomial in $n$ of degree at most $\ell-r+1$ for $n$ sufficiently large. We also prove $\text{SW}_p(k)$ is finite, which shows that summing these enumerations over all $w\in\text{SW}_p(k)$ yields a polynomial of degree at most $\ell-r+1$ equal to $B_p(n,k)$ for $n$ sufficiently large.

\begin{lemma}\label{short_words_constant_in_n}
    For any $k\in\ZZ^+$ and binary word $p$, the size of $\text{\normalfont SW}_p(k)$ is finite.
\end{lemma}

\begin{proof}
    Let $\ell$ be the length of $p$. Note that every binary word with at least $2\ell$ runs has at least one occurrence of $p$. Consequently, every binary word with at least $2\ell(k+1)$ runs has at least $k+1$ occurrences of $p$. So a binary word with exactly $k$ occurrences of $p$ has at most $2\ell(k+1)-1$ runs. As all runs of $0$'s and $1$'s in any word in $\text{SW}_{p}(k)$ are of length at most $k+r_{\max}-1$, the number of $0$'s and $1$'s in a word in $\text{SW}_{p}(k)$ is at most $(k+r_{\max}-1)(2\ell(k+1)-1)$. So the number of words consisting only of short runs of $0$'s and $1$'s with exactly $k$ occurrences of $p$ is less than $2^{(k+r_{\max}-1)(2\ell(k+1)-1)+1}$. Let $\mathcal{R}_p(k)$ be the set of such words.

    Note that every word in $\text{SW}_{p}(k)$ can be constructed by adding hatted letters to words in $\mathcal{R}_p(k)$. Let $w'\in\mathcal{R}_p(k)$. As no two hatted letters are adjacent, to construct a word in $\text{SW}_{p}(k)$ from $w'$, we have at most three choices between each space of $w'$ (including before the first letter of $w'$ and after the last letter of $w'$): insert a $\hatted{0}$, insert a $\hatted{1}$, or do not insert a hatted letter. Therefore, we have the bound
    \begin{align*}
        |\text{SW}_p(k)|\leq |\mathcal{R}_p(k)|\cdot 3^{\max_{w'\in \mathcal{R}_p(k)}(\text{length}(w'))+1}< 3\cdot6^{(k+r_{\max}-1)(2\ell(k+1)-1)+1},
    \end{align*}
    showing that $\text{\normalfont SW}_p(k)$ is finite.
    \end{proof}

We now prove every word in $\mathcal{B}_p(n,k)$ may be obtained by replacing the $\hatted{0}$'s and $\hatted{1}$'s in a word in $\text{SW}_p(k)$ with long runs of $0$'s and $1$'s respectively. To do this, we uniquely associate every word in $\mathcal{B}_{p}(n,k)$ to a word $\text{\normalfont SW}_p(k)$ by simply replacing all long runs of $0$'s with a single $\hatted{0}$ and all long runs of $1$'s with a single $\hatted{1}$.

\begin{lemma}\label{long_run_replacement_map}
    The procedure of replacing all long runs of $0$'s with a single $\hatted{0}$ and all long runs of $1$'s with a single $\hatted{1}$ is a map $f\colon\mathcal{B}_{p}(n,k)\to\text{\normalfont SW}_p(k)$.
\end{lemma}

\begin{proof}
    Let $W\in\mathcal{B}_p(n,k)$. If a long run of length $h$ in $W$ is used in an occurrence of $p$, then $W$ has at least 
    \begin{align*}
        \min_{\text{runs~}x^L\text{~of~}p}{k+r_{\max}\choose L}\geq k+1
    \end{align*}
    occurrences of $p$. So $c_p(f(W))=k$ and $f(W)$ satisfies property (\ref{property_4}) of Definition~\ref{SW_def}. By construction of $f$, property (\ref{property_1}) of Definition~\ref{SW_def} is satisfied by $f(W)$. Noting that at least one of any two consecutive runs in $W$ must be used in an occurrence of $p$, it follows that $W$ cannot have consecutive long runs. Thus, $f(W)$ satisfies properties (\ref{property_2}) and (\ref{property_3}) in Definition~\ref{SW_def}. Therefore $f(W)\in \text{\normalfont SW}_p(k)$.
\end{proof}

\begin{ex}
    \normalfont Let $p=001$ and $W=01111110101$. So $W\in\mathcal{B}_p(11,4)$. Since the length of $111111$ is $6\geq 4+2=k+r_{\max}$, we have $f(W)=0\hatted{1}0101\in \text{\normalfont SW}_p(4)$.
\end{ex}

\phantomsection\label{setting_star}
The next few lemmas and the proof of Theorem~\ref{polynomial_theorem} use the following notation, which we call ``setting \settingstar''. Let $k\in\ZZ^+$ and $p$ be a binary word of length $\ell$ with $r\geq2$ runs. Fix $x^L$ to be a run of length $L$ in $p$. Let $y=1-x$ and $u,u'$ be binary words (possibly empty) such that $p=ux^L u'$. Let $w\in\text{SW}_p(k)$. If $u,u'$ are nonempty, let $\mathbf{i}=i_1<\cdots<i_{c}$ be the first occurrence of $u$ in $w$ and $\mathbf{j}=j_1<\cdots<j_{d}$ be the last occurrence of $u'$ in $w$. If $u$ is empty, set $i_{c}=0$, and if $u'$ is empty, set $j_1=\text{length}(w)+1$. Lastly, define $Q_1(x^L)=Q_1=w_{1}\cdots w_{i_c}$ and $V(x^L)=V=w_{i_{c}+1}\cdots w_{j_1-1}$ and $Q_2(x^L)=Q_2=w_{j_1}\cdots w_{\text{length}(w)}$ so that we have the decomposition $w=Q_1VQ_2$.

\begin{ex}
    \normalfont Let $k=12$ and $p=1100111$ and $x^L$ be the run $00$. Then $y=1$ and $u=11$ and $u'=111$. Let $w=100110\hatted{1}0110011\in\text{\normalfont SW}_p(12)$. Then $\mathbf{i}$ is the occurrence described by the underlined letters and $\mathbf{j}$ is the occurrence described by the overlined letters in $\underline{1}00\underline{1}10\hatted{1}01\overline{1}00\overline{11}$. So $Q_1=1001$ and $V=10\hatted{1}01$ and $Q_2=10011$.
\end{ex}

The remaining lemmas in this section are used to show that there are at most $L-1$ hatted letters in $V(x^L)$ if $x^L$ is not at the boundary, at most $L$ hatted letters in $V(x^L)$ if $x^L$ is at the boundary, and that each hatted letter lies in exactly one $V(t^l)$ over all runs $t^l$ of $p$. We then sum over all the runs of $p$ to show that the total number of hatted letters in $w$ is at most $\ell-r+2$, from which we describe a balls-and-bins interpretation \cite[Section 1.9]{stanley_book} to prove that the number of words in $\mathcal{B}_p(n,k)$ that may be obtained by replacing the $\hatted{0}$'s and $\hatted{1}$'s in $w$ with long runs of $0$'s and $1$'s is a fixed polynomial of degree at most $\ell-r+1$ for $n$ sufficiently large.

%\begin{ex}\label{ex_for_L_minus_1}
  %  Let $k=126$ and $p=11111000001111$ and $x^L$ be the run $00000$. Then $y=1$ and $u=11111$ and $u'=1111$. Let $w=1001\hatted{0}11\hatted{0}11000\hatted{1}0\hatted{1}000111\hatted{0}1\hatted{0}\in\text{\normalfont SW}_p(126)$. Then $Q_1=1001\hatted{0}11\hatted{0}1$ and $V=1000\hatted{1}0\hatted{1}000$ and $Q_2=111\hatted{0}1\hatted{0}$. We now show that having a $\hatted{0}$ between $Q_1$ and $Q_2$ is impossible, exemplifying Lemma~\ref{at_most_L_minus_1}.

 %   If there was a $\hatted{0}$ between $Q_1$ and $Q_2$, then the word $1001111~0\cdots0~1111$ is necessarily a subsequence of some word obtained by expanding all hatted letters of $w$ into long runs, where $0\cdots0$ denotes the long run obtained from expanding the $\hatted{0}$ between $Q_1$ and $Q_2$. There are at least ${126+5\choose 5}>126$ occurrences of $00000$ in the long run, which implies that any word obtained by expanding all hatted letters of $w$ into long runs has more than $k$ occurrences of $p$, contradicting property (\ref{property_4}) of Definition~\ref{SW_def}. In fact, if we additionally regard the seven preexisting $0$'s in $w$, we can improve the lower bound on the occurrences of $00000$ in a word obtained by expanding all hatted letters of $w$ into long runs to ${7+126+5\choose 5}$.
%\end{ex}

\begin{lemma}\label{at_most_L_minus_1}
    Let notation be as in setting \settingstar. Then there are no $\hatted{x}$'s between $Q_1$ and $Q_2$, and the number of $\hatted{y}$'s between the first and last $x$ in $V$ is at most $L-1$.
\end{lemma}

\begin{proof}
    Let $\Lambda$ be the number of $x$'s in $V$. As $\mathbf{i}$ is the first occurrence of $u$ in $w$, $\mathbf{j}$ is the last occurrence of $u'$, and $k\geq1$, we know that $x^L$ occurs at least once in $V$. So $\Lambda\geq L$. Let $W$ be a word obtained by expanding the hatted letters of $w$ into long runs. For the sake of contradiction, suppose that there was an $\hatted{x}$ between $Q_1$ and $Q_2$. An occurrence of $p$ in $W$ may be formed by combining the first occurrence of $u$; an occurrence of $x^L$ after the first occurrence of $u$ and before the last occurrence of $u'$; and the last occurrence of $u'$. As the number of $x$'s in $W$ after the first occurrence of $u$ and before the last occurrence of $u'$ is at least $\Lambda+k+r_{\max}$ (by the existence of the $\hatted{x}$), we have that $W$ has at least ${\Lambda+k+r_{\max}\choose L}\geq{L+k+r_{\max}\choose L}>k$ occurrences of $p$, contradicting property (\ref{property_4}) of Definition~\ref{SW_def}. Thus, there are no $\hatted{x}$'s between $Q_1$ and $Q_2$.

    For the sake of contradiction, suppose there was an $\hatted{y}$ between the $(L+\lambda)$\textsuperscript{th} and $(L+\lambda+1)$\textsuperscript{th} $x$ in $V$ for some $\lambda\geq0$. Let $y^h$ be the long run in $W$ formed by expanding this $\hatted{y}$ and $\mathbf{j}'$ be the last occurrence of $u_2'\cdots u'_{\text{length}(u')}$ in $W$. An occurrence of $p$ in $W$ may be formed by combining the first occurrence of $u$, the first $L$ instances of $x$ after the first occurrence of $u$, a single $y$ from $y^h$ (note that $u_1'=y$), and finally $\mathbf{j}'$. Since the number of $y$'s to choose from in $y^h$ is at least $k+r_{\max}>k$, the number of occurrences of $p$ in $W$ is greater than $k$, contradicting property (\ref{property_4}) of Definition~\ref{SW_def}. By symmetry, there also cannot be an $\hatted{y}$ between the $(\Lambda-L-\lambda)$\textsuperscript{th} and $(\Lambda-L+1-\lambda)$\textsuperscript{th} $x$ in $V$ for any $\lambda\geq0$. As all runs of $\hatted{y}$'s are of length $1$ and have only $x$'s as surrounding letters, at most one $\hatted{y}$ may lie between the $s$\textsuperscript{th} and $(s+1)$\textsuperscript{th} $x$ in $V$. Furthermore, by the above, for a $\hatted{y}$ to exist between the $s$\textsuperscript{th} and $(s+1)$\textsuperscript{th} $x$ in $V$, it is necessary that $1\leq s\leq L-1$ and $\Lambda-L+1\leq s\leq \Lambda-1$. So the number of $\hatted{y}$'s between the first and last $x$ in $V$ is at most $\max((L-1)-(\Lambda-L+1)+1,0)=\max(2L-\Lambda-1,0)\leq L-1$.
\end{proof}

  %  \textcolor{red}{change this out with a figure}
   % Suppose that there was a $\hatted{1}$ between the second and third zero in $V$. Then the word $1001111100~1\cdots1~000001111$ is necessarily a subsequence of some word obtained by expanding all hatted letters of $w$ into long runs, where $1\cdots1$ denotes the long run obtained from expanding the $\hatted{1}$ between the second and third zero in $V$. Then we have the occurrences $\underline{1}00\underline{111}1100~\underline{1\cdots1}~\underline{000001111}$, where underlining the long run means we take one letter from the long run. Since there are at least $126+5>126$ letters in the long run, we contradict property (\ref{property_4}) of Definition~\ref{SW_def}.

Under the assumption that $u,u'$ are nonempty (i.e., when $x^L$ is not at the boundary), we now establish that there may be no $\hatted{y}$'s between $Q_1$ and $Q_2$ but not between two $x$'s in $V$.

\begin{lemma}\label{nonempty_no_hatted_ys}
     Let notation be as in setting \settingstar. If $u,u'$ are nonempty, then there are no $\hatted{y}$'s between $Q_1$ and the first $x$ in $V$, nor between the last $x$ in $V$ and $Q_2$.
\end{lemma}

\begin{proof}
    Note that $u$ ending with an $x$ would contradict the maximality of the run $x^L$ in $p$. So $u$ must end with a $y$. By symmetry, $u'$ must also begin with a $y$. There cannot be a $\hatted{y}$ between $Q_1$ and the first $x$ in $V$, otherwise we would contradict either property (\ref{property_2}) or property (\ref{property_3}) of Definition~\ref{SW_def}. By the same reasoning, there also cannot be a $\hatted{y}$ between the last $x$ in $V$ and $Q_2$.
\end{proof}

%Note that $u$ ending with an $x$ would contradict the maximality of the run $x^L$ in $p$. So $u$ must end with a $y$. By symmetry, $u'$ must also begin with a $y$. For the sake of contradiction suppose there was a $\hatted{y}$ between $w_{i_c}$ and $w_{i_c+1}$. If we replace this $\hatted{y}$ between $w_{i_{c}}$ and $w_{i_{c}+1}$ with a long run of $y$'s, then $w_{i_1},\dots,w_{i_{c-1}}$ followed by a $y$ from the long run replacing this $\hatted{y}$, an occurrence of $x^L$ in $V$, and finally by the last occurrence of $u'$ yields an occurrence of $p$. Since the number of $y$'s to choose from in the long run is at least $k+r_{\max}$, the number of occurrences of $p$ in any word obtained by expanding this $\hatted{y}$ into a long run of $y$'s is greater than $k$, contradicting property (iv) of Definition~\ref{SW_def}. Thus, there are no $\hatted{y}$'s between $w_{i_{c}}$ and $w_{i_{c}+1}$. By symmetry, there are also no $\hatted{y}$'s between $w_{j_1-1}$ and $w_{j_1}$.

We now consider the cases where exactly one of $u,u'$ is empty, establishing that these boundary cases permit at most one more $\hatted{y}$ between $Q_1$ and $Q_2$ than the case where $u,u'$ are nonempty.

\begin{lemma}\label{empty_cases}
     Let notation be as in setting \settingstar. If $x^L$ is the first run of $p$, then there is at most one $\hatted{y}$ before $V$ and no $\hatted{y}$'s between the last $x$ in $V$ and $Q_2$. If $x^L$ is the last run of $p$, then there is at most one $\hatted{y}$ after $V$ and no $\hatted{y}$'s between $Q_1$ and the first $x$ in $V$.
\end{lemma}

\begin{proof}
    Since $r\geq2$, it follows that $x^L$ cannot be the first and last run of $p$. Suppose that $x^L$ is the first run of $p$. Then there are only $y$'s or a single $\hatted{y}$ before $w_{i_c+1}=w_1$. As $u$ is empty, $u'$ is nonempty. So if there were a $\hatted{y}$ between the last $x$ in $V$ and $Q_2$, we would obtain the same contradiction as in Lemma~\ref{nonempty_no_hatted_ys}. By symmetry, the case where $x^L$ is the last run of $p$ follows.
    %Suppose that $u'$ is empty. So $x^L$ is the last run of $p$, implying that there are only $y$'s or a single $\hatted{y}$ after $w_{j_1-1}=w_{\text{length}(w)}$. As $u'$ is empty, $u$ is nonempty. So if there were a $\hatted{y}$ between $Q_1$ and the first $x$ in $V$, we would obtain the same contradiction as in Lemma~\ref{nonempty_no_hatted_ys}.
\end{proof}

To bound the total number of hatted letters in $w$, we show that every hatted letter in $w$ lies in exactly one $V(t^l)$. To prove this fact, we use the following notation: given two words $\nu,\nu'$, we write that $\nu'\subseteq \nu$ if $\nu'=\nu_{h_1}\cdots\nu_{h_2}$ for some $1\leq h_1\leq h_2\leq\text{length}(\nu)$.

\begin{lemma}\label{exactly_one_Q_1_Q_2_pair}
    Let notation be as in setting \settingstar. Then every hatted letter in $w$ lies in exactly one word in $\{V(t^l): t^l\text{\normalfont~is~a~run~of~}p\}$.
\end{lemma}

\begin{proof}
    Let $\hatted{a}$ be a hatted letter of $w$. If this $\hatted{a}$ is at the beginning of $w$, then this $\hatted{a}$ lies only in $V(\beta_1^{d_1})$, where $\beta_1^{d_1}$ is the first run of $p$. If this $\hatted{a}$ is at the end of $w$, then this $\hatted{a}$ lies only in $V(\beta_2^{d_2})$, where $\beta_2^{d_2}$ is the last run of $p$.
    
    Thus, assume that this $\hatted{a}$ is in the middle of $w$. Let $b=1-a$. Then by properties (\ref{property_2}) and (\ref{property_3}) of Definition~\ref{SW_def}, this $\hatted{a}$ must be surrounded only by $b$'s. For the sake of contradiction, suppose that there is no occurrence of $p$ using both of the surrounding $b$'s. Let $W$ be a word obtained by expanding the hatted letters of $w$ into long runs. Then, since at least one of any two consecutive runs must be used in an occurrence of $p$, this long run of $a$'s must be used in an occurrence of $p$ in $W$. This implies that $c_p(W)>k$, contradicting property (\ref{property_4}) of Definition~\ref{SW_def}. So both of the $b$'s surrounding this $\hatted{a}$ are used in an occurrence $\mathbf{g}=g_1<\cdots< g_{\ell}$ of $p$ in $w$. Let $e$ be such that $w_{g_e}$ and $w_{g_{e+1}}$ equal the two $b$'s surrounding this $\hatted{a}$. Let $b^L$ be the run of $p$ that $w_{g_e}$ and $w_{g_{e+1}}$ constitute. Then $Q_1(b^L)$ precedes $w_{g_e}$ and $w_{g_{e+1}}$ and $Q_2(b^L)$ follows $w_{g_e}$ and $w_{g_{e+1}}$. Therefore, this $\hatted{a}$ lies in $V(b^L)$. See Example~\ref{exactly_one_part_one_ex} for a diagrammed example.
    
    For the sake of contradiction, suppose that this $\hatted{a}$ lies between both the pairs $Q_1(x_1^{L_1}),Q_2(x_1^{L_1})$ and $Q_1(x_2^{L_2}),Q_2(x_2^{L_2})$ for some distinct runs $x_1^{L_1},x_2^{L_2}$ of $p$. If $x_1\neq x_2$, then $\hatted{a}=\hatted{x_1}$ or $\hatted{a}=\hatted{x_2}$, contradicting Lemma~\ref{at_most_L_minus_1} in either case. Thus, assume that $x_1=x_2=1-a$. Then there exists a run $y^{u}$ in $p$ between $x_1^{L_1}$ and $x_2^{L_2}$, where $y=1-x_1=a$. Without loss of generality, suppose that $x_1^{L_1}$ precedes $x_2^{L_2}$ in $p$. As $y^u$ lies between $x_1^{L_1}$ and $x_2^{L_2}$, we have $Q_1(x_1^{L_1})\subseteq Q_1(y^u)\subseteq Q_1(x_2^{L_2})$ and $Q_2(x_2^{L_2})\subseteq Q_2(y^u)\subseteq Q_2(x_1^{L_1})$. So this $\hatted{a}$ lies between $Q_1(y^u)$ and $Q_2(y^u)$. Therefore, by Lemma~\ref{at_most_L_minus_1}, $a\neq y$, a contradiction. See Example~\ref{exactly_one_part_two_ex} for a diagrammed example.
\end{proof}

\begin{ex}\label{exactly_one_part_one_ex}
    \normalfont In the setting of the proof of Lemma~\ref{exactly_one_Q_1_Q_2_pair}, we provide an example showing that $\hatted{a}$ lying in $V(b^L)$ follows by the surrounding $b$'s being used in an occurrence of $p$ in $w$. 
    
    Let $p=1100111$ and $w=100110\hatted{1}0110011$. Selecting $\hatted{1}$ to be the $\hatted{a}$, the surrounding $b$'s are $w_6$ and $w_8$. Consider the following diagram displaying an occurrence $\mathbf{g}$ that uses both $w_6$ and $w_8$, showing how we select the run $b^L$.
    \begin{align*}
        \mathbf{g}&=\phantom{1~0~0}~1~1\hspace{0.25em}\overbrace{0~\phantom{\hatted{1}}~0}^{b^L}~1~1~\phantom{0~0}~1~\phantom{1}\\
        w&=1~0~0~1~1~{0}~{\hatted{1}}~{0}~1~1~0~0~1~1
    \end{align*}
    So in this case, $b^L=0^2$. Thus, we have the decomposition
    \begin{align*}
        w=\underbrace{1~0~0~1}_{Q_1(b^L)}~\Big|~\underbrace{1~{0}~{\hatted{1}}~{0}~1}_{V(b^L)}~\Big|~\underbrace{1~0~0~1~1}_{Q_2(b^L)}.
    \end{align*}
\end{ex}

\begin{ex}\label{exactly_one_part_two_ex}
    \normalfont In the setting of the proof of Lemma~\ref{exactly_one_Q_1_Q_2_pair}, we provide an example showing the inclusions $Q_1(x_1^{L_1})\subseteq Q_1(y^u)\subseteq Q_1(x_2^{L_2})$ and $Q_2(x_2^{L_2})\subseteq Q_2(y^u)\subseteq Q_2(x_1^{L_1})$.
    
    Let $p=0^2 1^5 0^5 1^4 0^3$ and $w=001001\hatted{0}11\hatted{0}11000\hatted{1}0\hatted{1}000111\hatted{0}1000$. Let $1^5$ and $1^4$ be the distinct runs $x_1^{L_1},x_2^{L_2}$ of $p$. Since they are distinct, we have the run $0^5=y^u$ that lies in between $x_1^{L_1}$ and $x_2^{L_2}$. To observe that $Q_1(1^5)\subseteq Q_1(0^5)\subseteq Q_1(1^4)$ and $Q_2(1^4)\subseteq Q_2(0^5)\subseteq Q_2(1^5)$, consider the following diagram.
    \begin{align*}
        w\phantom{aaaaaa}&\phantom{\rightarrow}0~0~1~0~0~1~\hatted{0}~1~1~\hatted{0}~1~1~0~0~0~\hatted{1}~0~\hatted{1}~0~0~0~1~1~1~\hatted{0}~1~0~0~0\\
        Q_1(1^5)\phantom{aa}Q_2(1^5)&\phantom{\rightarrow}0~0\phantom{~1~0~0~1~\hatted{0}~1~1~\hatted{0}~1~1~0~0~}0~\hatted{1}~0~\hatted{1}~0~0~0~1~1~1~\hatted{0}~1~0~0~0\\
        Q_1(0^5)\phantom{aa}Q_2(0^5)&\phantom{\rightarrow}0~0~1~0~0~1~\hatted{0}~1~1~\hatted{0}~1\phantom{~1~0~0~0~\hatted{1}~0~\hatted{1}~0~0~0}~1~1~1~\hatted{0}~1~0~0~0\\
        Q_1(1^4)\phantom{aa}Q_2(1^4)&\phantom{\rightarrow}0~0~1~0~0~1~\hatted{0}~1~1~\hatted{0}~1~1~0~0~0~\hatted{1}~0~\hatted{1}~0\phantom{~0~0~1~1~1~\hatted{0}~1}~0~0~0.
    \end{align*}
\end{ex}

Lastly, we state the formula for $B_p(n,1)$ provided by Menon and Singh in \cite{subseq_freq} for the reader to refer to, as it is frequently invoked in the proof of the second part of Theorem~\ref{polynomial_theorem}.

\begin{prop}[{\cite[Proposition 3.2]{subseq_freq}}]\label{menon_singh_result_k_1}
    Let $n\geq0$ and $p$ be a binary word of length $\ell$ with $r$ runs. Then
    \begin{align*}
        B_p(n,1)={n-r+1\choose\ell-r+1}.
    \end{align*}
\end{prop}

\begin{proof}[\textbf{Proof of Theorem~\ref{polynomial_theorem}}]

    Let $f$ be the long-run-replacement map as in Lemma~\ref{long_run_replacement_map}. Consequently, we write
    \begin{align}\label{preimage_equation}
        B_p(n,k)=\sum_{w\in\text{SW}_p(k)}\left|f^{-1}(\{w\})\right|,
    \end{align}
    which by Lemma~\ref{short_words_constant_in_n} is a finite sum. We proceed by showing for $n$ sufficiently large that the number of words in each $f^{-1}(\{w\})$ is a polynomial in $n$ of degree at most $\ell-r+1$.

    %By Lemma~\ref{at_most_L_minus_1}, all hatted letters in each $V(x^L)$ are only $\hatted{y}$'s. Thus, 

    %the set of binary words of length $n$ with exactly $k$ occurrences of $p$ that may be formed from replacing each hatted letter $\hatted{a}$ in $w$ with a long run of $a$'s

    Let $h_w$ denote the exact number of hatted letters in $w\in\text{SW}_p(k)$. We can use a balls-and-bins interpretation to enumerate the size of $f^{-1}(\{w\})$ in terms of $h_w$ by considering the bins to be the hatted letters, and the balls, counting from $k+r_{\max}$, to be the $a$'s replacing the $\hatted{a}$'s. That is, if a hatted letter $\hatted{a}$ is replaced by $a^r$, then we interpret this to be $r-(k+r_{\max})$ balls since each hatted letter needs to be replaced by a long run. Since we have $n-(\text{length}(w)-h_w)-h_w(k+r_{\max})$ balls and $h_w$ bins,
    \begin{align}\label{size_of_preimage}
        \left|f^{-1}(\{w\})\right|={n-\text{length}(w)-h_w(k+r_{\max}-2)-1\choose h_w-1}.
    \end{align}
     Thus, \eqref{preimage_equation} together with \eqref{size_of_preimage} give that
    \begin{align*}
         B_p(n,k)=\sum_{w\in\text{SW}_p(k)} {n-\text{length}(w)-h_w(k+r_{\max}-2)-1\choose h_w-1},
    \end{align*}
    which is a fixed polynomial of degree $\max_{w\in\text{SW}_p(k)}h_w-1$ for $n$ sufficiently large.
    
    % To avoid the number of occurrences of $p$ exceeding $k$, we claim that
  %  \begin{enumerate}[(i)]
  %      \item no $m$'s between $\mathbf{i}$ and $\mathbf{j}$ can be replaced by a long run of $x$'s,
    %    \item no $m$'s between $w_{i_{c}}$ and $w_{i_{c}+1}$ (assuming $u$ is nonempty) nor $w_{j_1-1}$ and $w_{j_1}$ (assuming $u'$ is nonempty) can be replaced by a long runs of $y$'s, and
    %    \item there are at most $L-1$ $m$'s in $V$, which by (i) can be replaced only by a long run of $y$'s.
  %  \end{enumerate}
   % Note that (i) and (ii) imply that, assuming $u,u'$ are nonempty, there are no $m$'s between $w_{i_{c}}$ and $w_{i_{c}+1}$ nor $w_{j_1-1}$ and $w_{j_1}$.
    
%    Thus, to avoid the number of occurrences exceeding $k$, for nonempty $u,u'$, there are at most $L-1$ $m$'s between $u$ and $u'$, and each can be replaced with a long run of only one letter type. 

Let notation be as in setting \settingstar. By Lemmas~\ref{at_most_L_minus_1},~\ref{nonempty_no_hatted_ys}, and~\ref{empty_cases}, the number of hatted letters in $V(x^L)$ is at most $L$ if $x^L$ is at the boundary of $p$, and is at most $L-1$ if $x^L$ is not at the boundary of $p$. So by Lemma~\ref{exactly_one_Q_1_Q_2_pair}, $h_w$ is at most
    \begin{align*}
        \sum_{\text{run $B$ of $p$ at the boundary}}\text{length}(B)+\sum_{\text{run $R$ of $p$ not at the boundary}}(\text{length}(R)-1)=\ell-r+2,
    \end{align*}
proving the first part of the theorem.

We now prove the second part of the theorem. Let $r_1$ be the number of runs of length $1$ in $p$, and for any $b\geq2$, let $r_{b,\partial}$ be the number of runs of length $b$ that are either the first or last run in $p$. Leveraging the existence of these two types of runs in $p$, we construct two families of binary words of length $n$ with exactly $k$ occurrences of $p$, each of which of size $\Theta(n^{\ell-r+1})$, showing that $\deg(B_p(n,k))=\ell-r+1$ for $n$ sufficiently large. Since the result is true for $k=1$, we may take $k\geq2$ to ensure that the two families are disjoint.

    The following construction generalizes \cite[Construction 3.1]{subseq_freq}. Let $p_j$ be a run of size $1$ in $p$ and $w'$ be a word of length $n-k+1$ with exactly one occurrence of $p$ at $i_1<\cdots<i_{\ell}$, which exists by Proposition~\ref{menon_singh_result_k_1} since $n$ is sufficiently large. Let $w$ be the word of length $n$ obtained by replacing $w'_{i_j}$ with $k$ copies of $p_j$. By construction, $c_p(w)=k$. Since there were $r_1$ choices for $p_j$, it follows from Proposition~\ref{menon_singh_result_k_1} that the number of ways to make this construction is $r_1{n-r+2-k\choose \ell-r+1}$. 
    %Proceeding by adding $k-1$ copies of $p_j$ immediately after $w_{i_j}'$ yields a word $w$ of length $n$ such that $c_p(w)=k$. 

    The following construction generalizes \cite[Construction 3.2]{subseq_freq}. Let $b\in\ZZ^+$ be such that ${b+s\choose b}=k$ for some $s\geq0$. Let $p_jp_{j+1}\cdots p_{j+b-1}$ be a run of size $b$ at the boundary of $p$ and $w'$ be a word of length $n-s$ with exactly one occurrence of $p$ at $i_{1}<\cdots<i_{\ell}$, which exists by Proposition~\ref{menon_singh_result_k_1} since $n$ is sufficiently large. If $j=1$ (so the run is the first run of $p$), adding $s$ copies of $p_1$ immediately after $w'_{i_b}$ yields a word $w$ of length $n$ such that $c_p(w)={b+s\choose b}=k$. If $j=\ell-b+1$ (so the run is the last run of $p$), adding $s$ copies of $p_{\ell}$ immediately before $w'_{i_{\ell-b+1}}$ yields a word $w$ of length $n$ such that $c_p(w)={b+s\choose b}=k$. Since there were $r_{b,\partial}$ choices for $p_jp_{j+1}\cdots p_{j+b-1}$, it follows from Proposition~\ref{menon_singh_result_k_1} that the number of ways to make this construction is $r_{b,\partial}{n-r-s+1\choose \ell-r+1}$. From these two constructions, it follows that
    \begin{align*}
        B_p(n,k)\geq r_1{n-r+2-k\choose \ell-r+1}+\sum_{b}r_{b,\partial}{n-r-s+1\choose \ell-r+1},
    \end{align*}
    where the sum is taken over all $b\geq2$ such that ${b+s\choose b}=k$ for some $s\geq0$. Thus, for $n$ sufficiently large, $\deg(B_p(n,k))=\ell-r+1$ since $r_1+\sum_{b} r_{b,\partial}>0$ by assumption.
\end{proof}

%to be the set of compositions of $0,1,\dots,R$ with exactly $r$ parts, each of which is required to be odd.

%The formula, stated in Proposition~\ref{combinatorial_formula_for_number_of_occurrences}, is a sum of products of differences of binomial coefficients, where the sum is taken over a class of restricted integer compositions\footnote{A \emph{composition} of an integer $n\geq0$ is a tuple of positive integers whose elements sum to $n$.} and each product is taken over the parts of a set partition (dependent on the composition) of the runs of $w$. 

 %The correspondence between this formula and this algorithm is expounded upon immediately following the statement of Proposition~\ref{combinatorial_formula_for_number_of_occurrences}. 

\section{Proof of Theorem~\ref{upper_bound_thm}}\label{section_3}

To establish the claimed bound, we first prove an exact combinatorial formula for $c_p(w)$ conducive to being bound from above, which merely algebraically encodes a natural algorithm for computing occurrences of $p$ by finding occurrences of the individual runs of $p$ in $w$. As we later show that the proof of the formula reduces to the case where $w$ has no NBRs, assume without loss of generality that $w$ has no NBRs. Letting $r$ denote the number of runs in $p$, the aforementioned algorithm for computing $c_p(w)$ can be loosely described by the following:
\begin{enumerate}[(i)]
    \item\label{1_algorithm} partition a prefix\footnote{A \emph{prefix} of $w$ is a subsequence of the form $w_1\cdots w_i$ for $1\leq i\leq \text{length}(w)$.} of $w$ into $r$ intervals in a particular manner to be described later;
    \item\label{2_algorithm} for each $1\leq i\leq r$, enumerate the occurrences of the $i$\textsuperscript{th} run of $p$ in the $i$\textsuperscript{th} interval of the partition such that at least one letter of the $i$\textsuperscript{th} run of $p$ is taken from the last run of the $i$\textsuperscript{th} interval;
    \item\label{3_algorithm} take the product of the enumerations in (\ref{2_algorithm}) over all $1\leq i\leq r$ to obtain the number of occurrences of $p$ in $w$ given the constraint of the given partition; and
    \item\label{4_algorithm} sum these products over all such partitions of prefixes of $w$.
\end{enumerate}

By construction, we show that each occurrence of $p$ in $w$ corresponds to exactly one such partition of a prefix of $w$; in the language of relations, the relation on the occurrences of $p$ in $w$ defined by $\mathbf{i}\sim\mathbf{j}$ if $\mathbf{i},\mathbf{j}$ both correspond to the same partition of a prefix of $w$ is an equivalence relation. The product described in (\ref{3_algorithm}) is then the size of each equivalence class under this relation.

Letting $R'$ be the number of runs of $w$ with all NBRs removed, we prove that these equivalence classes can be represented as a subset of the set of $r$-tuples $(c_1,\dots,c_r)$ of positive odd integers such that $\sum_{j=1}^{r}c_j\leq R'$. The exact value of $\sum_{j=1}^{r}c_j$ determines which prefix of $w$ is partitioned as in step (\ref{1_algorithm}) of the algorithm. That is, the prefix of $w=y_1^{v_1}\cdots y_s^{v_R}$ that we partition is $y_1^{v_1}\cdots y_s^{v_s}$, where $s=\sum_{j=1}^{r}c_j$.

\begin{definition}
    Let $r,R\in\ZZ^+$. We define $\mathscr{C}_{r,\text{\normalfont odd}}^{\leq R}$ to be the set of $r$-tuples $(c_1,\dots,c_r)$ of positive odd integers such that $\sum_{j=1}^{r}c_i\leq R$.
\end{definition}

We illustrate in the following example how to compute the corresponding tuple in $\mathscr{C}_{r,\text{\normalfont odd}}^{\leq R'}$ of an occurrence of $p$ in $w$. The reverse of this process, computing all occurrences corresponding to a given tuple in $\mathscr{C}_{r,\text{\normalfont odd}}^{\leq R'}$, is the same as steps (\ref{1_algorithm}), (\ref{2_algorithm}), and (\ref{3_algorithm}) of the algorithm.

\begin{ex}\label{inverse_of_algorithm}
    \normalfont Let $p=111001111$ and $w=10011110101001101110101$. Let $\mathbf{i}$ be the occurrence of $p$ described by the underlined letters in $\underline{1}00\underline{11}1101\underline{0}10\underline{011}0\underline{11}10101$. Begin by placing a bar in $\mathbf{i}$ immediately after the last letter of each run of $p$, obtaining
    \begin{align*}
        \underline{1}00\underline{11}\Big|1101\underline{0}10\underline{0}\Big|\underline{11}0\underline{11}\Big|10101.
    \end{align*}
    Now, slide each bar to the right until it does not lie in the middle of a single run of $w$. If a bar already does not lie in a single run of $w$ (e.g., the second bar above), then do not slide that bar. Applying this procedure to the above yields
    \begin{align*}
        \underline{1}00\underline{11}11\Big|01\underline{0}10\underline{0}\Big|\underline{11}0\underline{11}1\Big|0101,
    \end{align*}
    from which we obtain the partition $1001111,010100,110111$ of the prefix $1001111010100110111$ of $w$ into $r=3$ intervals. Finally, the corresponding tuple of $\mathbf{i}$ is then the number of runs in each interval of this partition, returning $(3,5,3)\in \mathscr{C}_{3,\text{\normalfont odd}}^{\leq 15}$. 
\end{ex}

\begin{comment}

The aforementioned algorithm essentially amounts to doing the inverse of the process illustrated in Example~\ref{inverse_of_algorithm}; that is, finding the number of occurrences of $p$ in $w$ corresponding to a given tuple in $\mathscr{C}_{r,\text{\normalfont odd}}^{\leq R'}$, as exemplified by the following.

\begin{ex}
    Let $p$ and $w$ be as in Example~\ref{inverse_of_algorithm} and $(3,5,3)\in \mathscr{C}_{3,\text{\normalfont odd}}^{\leq 15}$. Begin by placing bars after the $3$\textsuperscript{rd}, $(3+5)\textsuperscript{th}$, and $(3+5+3)$\textsuperscript{th} runs in $w$, obtaining
    \begin{align*}
        {1}00{11}11\Big|01{0}10{0}\Big|{11}0{11}1\Big|0101.
    \end{align*}
    Now apply step (ii) to this partition. Computing
\end{ex}

\end{comment}

Given a $(c_1,\dots,c_r)=C\in \mathscr{C}_{r,\text{\normalfont odd}}^{\leq R'}$ corresponding to a unique equivalence class, we now precisely describe the $r$ subsets $A_i\subseteq\{1,\dots,R'\}$ determining the runs of $w$ where the $i$\textsuperscript{th} run of $p$ may occur in given the restriction of $C$. In short, $C$ determines which runs of $w$ we are using to construct an occurrence of $p$.

\begin{definition}
    Given $C=(c_1,\dots,c_{r})\in\mathscr{C}_{r,\text{\normalfont odd}}^{\leq R}$, let $C([R])\defeq\{A_1,\dots,A_{r}\}$, where
    \begin{align*}
        A_i=\left\{1+\sum_{j=1}^{i-1}c_j,\dots,\sum_{j=1}^{i}c_j\right\}\cap(2\ZZ+i).
    \end{align*}
\end{definition}

We intersect with $2\ZZ+i$ since the letters comprising the $i$\textsuperscript{th} run of $p$ are equal only to the letters in $j$\textsuperscript{th} run of $w$ for any $j$ of the same parity as $i$. In Example~\ref{inverse_of_algorithm} where $C=(3,5,3)$,
\begin{align*}
    C([15])=\{\{1,3\},\{4,6,8\},\{9,11\}\}
\end{align*}
since we can use the first and third runs in the first interval to make $111$; the fourth, sixth, and eighth runs in the second interval to make $00$; and the ninth and eleventh runs in the third interval to make $1111$.

We now introduce some notation for Proposition~\ref{combinatorial_formula_for_number_of_occurrences} and its proof. Recall that for two integers $i,j$, the Kronecker delta $\delta_{i,j}$ is $1$ if $i=j$, and is $0$ otherwise. Given binary words $p$ and $w$, we define $\mathfrak{C}_p(w)$ to be the set of occurrences of $p$ in $w$, wherein we have $|\mathfrak{C}_p(w)|=c_p(w)$.

\begin{prop}\label{combinatorial_formula_for_number_of_occurrences}
    Let $p=x_1^{u_1}\cdots x_r^{u_r}$ and $w=y_1^{v_1}\cdots y_{R}^{v_R}$ be binary words. Let $R'$ be the number of runs of $w$ which are not NBRs and set $\epsilon=1-\delta_{x_1,y_1}$. Then
    \begin{align}\label{eqn_comb_formula}
        c_p(w)=\sum_{C\in\mathscr{C}_{r,\text{\normalfont odd}}^{\leq R'}}\prod_{A_i\in C([R'])}\left({\sum_{j\in A_i}v_{j+\epsilon}\choose u_i}-{\sum_{j\in A_i\setminus\{\max(A_i)\}}v_{j+\epsilon}\choose u_i}\right).
    \end{align}
\end{prop}

With the formal statement of our combinatorial formula written, we can now see exactly how Proposition~\ref{combinatorial_formula_for_number_of_occurrences} algebraically encodes the algorithm described at the beginning of this section:
\begin{enumerate}[(i)]
    \item selecting a $C\in\mathscr{C}_{r,\text{\normalfont odd}}^{\leq R'}$ corresponds to step (\ref{1_algorithm});
    \item the factors in the product are the enumerations described in step (\ref{2_algorithm});
    \item step (\ref{3_algorithm}) is equivalent to the product; and
    \item summing over all $C\in\mathscr{C}_{r,\text{\normalfont odd}}^{\leq R'}$ is as described in step (\ref{4_algorithm}).
\end{enumerate}

\begin{proof}[Proof of Proposition~\ref{combinatorial_formula_for_number_of_occurrences}]
    We first note that the proof reduces to the case where $w$ has no NBRs, i.e., $x_1=y_1$ and $x_r=y_R$. Indeed, letting $w'=y_{1+\epsilon}^{v_{1+\epsilon}}\cdots y_{R'+\epsilon}^{v_{R'+\epsilon}}$, it follows that $c_p(w)=c_p(w')$ and $w'$ has no NBRs.
    
    Thus, restricting to the case where $x_1=y_1$ and $x_r=y_R$, let $\mathbf{i}=(i_1<\cdots<i_{\ell})\in\mathfrak{C}_p(w)$. For each $1\leq a\leq r$, let $s(a)=\sum_{j=1}^{a}u_j$ and $I_a=i_{s(a)}$ and $g_{\mathbf{i}}(a)$ be such that $w_{I_a}$ lies in the run $y_{g_{\mathbf{i}}(a)}^{v_{g_{\mathbf{i}}(a)}}$. In words, $g_{\mathbf{i}}(a)$ is the index of the run in $w$ which contains the last letter of the $a$\textsuperscript{th} run of $p$ with respect to $\mathbf{i}$. In Example~\ref{inverse_of_algorithm}, we have $g_{\mathbf{i}}(1)=3$, $g_{\mathbf{i}}(2)=3+5=8$, and $g_{\mathbf{i}}(3)=3+5+3=11$. 
    
    Define $C=(g_{\mathbf{i}}(1),g_{\mathbf{i}}(2)-g_{\mathbf{i}}(1),g_{\mathbf{i}}(3)-g_{\mathbf{i}}(2),\dots,g_{\mathbf{i}}(r)-g_{\mathbf{i}}(r-1))$ to be the $r$-tuple associated to $\mathbf{i}$. We claim that each element of $C$ is odd. As $x_1=y_1$, an occurrence of an odd-indexed run of $p$ in $w$ must lie only in the odd-indexed runs of $w$. Similarly, an occurrence of an even-indexed run of $p$ in $w$ must lie only in the even-indexed runs of $w$. So $g_{\mathbf{i}}(n)$ is of the same parity as $n$. Thus, each element of $C$ is odd. Observing that the sum of the elements of $C$ telescopes to $g_{\mathbf{i}}(r)\leq R$, we have
    \begin{align*}
        \{r\text{-tuple~associated~to~}\mathbf{i}:\mathbf{i}\in\mathfrak{C}_{p}(w)\}\subseteq\mathscr{C}_{r,\text{odd}}^{\leq R}.
    \end{align*}
    So
    \begin{align}\label{partition_of_occurrences}
        c_p(w)=\sum_{C\in\mathscr{C}_{r,\text{odd}}^{\leq R}}\left|\{\mathbf{i}\in\mathfrak{C}_{p}(w):C=r\text{-tuple~associated~to~}\mathbf{i}\}\right|.
    \end{align}
    To complete the proof, we now enumerate the summands in \eqref{partition_of_occurrences}. Let $(c_1,\dots,c_r)=C\in \mathscr{C}_{r,\text{odd}}^{\leq R}$ and $\mathbf{i}\in\mathfrak{C}_{p}(w)$ be such that $C$ equals the $r\text{-tuple associated to }\mathbf{i}$. Let $1\leq t\leq r$. For notational convenience, set $g_{\mathbf{i}}(0)=0$. 
    
    %By construction, the last letter of the run $x_t^{u_t}$ occurs in the run $y_{g_{\mathbf{i}}(t)}^{v_{g_{\mathbf{i}}(t)}}$. Furthermore, the first letter of the run $x_t^{u_t}$ occurs in a run after the run $y_{g_{\mathbf{i}}(t-1)}^{v_{g_{\mathbf{i}}(t-1)}}$. 
    
    By construction, the last letter of the $t$\textsuperscript{th} run of $p$ occurs in the $g_{\mathbf{i}}(t)$\textsuperscript{th} run of $w$. Furthermore, the first letter of the $t$\textsuperscript{th} run of $p$ occurs in a run after the $g_{\mathbf{i}}(t-1)$\textsuperscript{th} run of $w$. Thus, as $x_t=y_b$ only if $b$ has the same parity as $t$, the $t$\textsuperscript{th} run of $p$ may only occur in runs with indices in $\{g_{\mathbf{i}}(t-1)+1,g_{\mathbf{i}}(t-1)+3\dots,g_{\mathbf{i}}(t)-2,g_{\mathbf{i}}(t)\}$, with at least one letter occurring in the $g_{\mathbf{i}}(t)$\textsuperscript{th} run of $w$. Noting that $\sum_{h=1}^{T}c_h=g_{\mathbf{i}}(T)$, the occurrences of the $t$\textsuperscript{th} run of $p$ with the restriction of $C$ are enumerated by the $u_t$-element subsets of
    \begin{align}\label{big_set_prop_proof}
        \left\{1,\dots,v_{g_{\mathbf{i}}(t-1)+1}+v_{g_{\mathbf{i}}(t-1)+3}+\cdots+v_{g_{\mathbf{i}}(t)-2}+v_{g_{\mathbf{i}}(t)}\right\}=\left\{1,\dots,\sum_{j\in A_t} v_j\right\}
    \end{align}
    containing at least one of the $v_{g_{\mathbf{i}}(t)}=v_{\max(A_t)}$ largest elements of \eqref{big_set_prop_proof}. The number of these subsets is
    \begin{align*}
        {\sum_{j\in A_t}v_{j}\choose u_t}-{\sum_{j\in A_t\setminus\{\max(A_t)\}}v_{j}\choose u_t},
    \end{align*}
    which may be seen via an all-minus-bad strategy. Thus, the size of $\{\mathbf{i}\in\mathfrak{C}_{p}(w):C=r\text{-tuple~associated~to~}\mathbf{i}\}$ is
    \begin{align}\label{complicated_summand}
        \prod_{A_t\in C([R])}\left({\sum_{j\in A_t}v_{j}\choose u_t}-{\sum_{j\in A_t\setminus\{\max(A_t)\}}v_{j}\choose u_t}\right),
    \end{align}
    yielding \eqref{eqn_comb_formula}.
\end{proof}

We remark that another combinatorial formula for $c_p(w)$ was given recently by Tian in \cite[Theorem 4.0.6]{other_formula_for_occurrence_count}. Counting between the first and last occurrence of $p$ in $w$, Tian's formula uses nested summations (one for each run of $p$) to determine exactly which runs of $w$ are used to make an occurrence of $p$. As occurrences are intricate objects to enumerate, the moral difference between Proposition~\ref{combinatorial_formula_for_number_of_occurrences} and \cite[Theorem 4.0.6]{other_formula_for_occurrence_count} lies in where we decide to bake in the structural complexity of $\mathfrak{C}_p(w)$. In Tian's formula, the indexing sets of the nested summations are complicated, allowing the summands to be a simple product of nonzero binomial coefficients. In contrast, our formula has complicated summands, i.e., \eqref{complicated_summand}, which may be zero, allowing for the indexing set $\mathscr{C}_{r,\text{\normalfont odd}}^{\leq R'}$ of the summation to be easily enumerated.

Constructing Proposition~\ref{combinatorial_formula_for_number_of_occurrences} in this contrasting manner allows us to tractably obtain an upper bound on $c_p(w)$. With the advent of our combinatorial formula, the proof of Theorem~\ref{upper_bound_thm} simply rests on the fact that
\begin{align*}
    \sum_{s\in S}a_s\leq |S|\max_{s\in S}a_s
\end{align*}
for a finite set $S$ and a sequence $\{a_s\}_{s\in S}$. We proceed by enumerating the size of $\mathscr{C}_{r,\text{\normalfont odd}}^{\leq R}$.

\begin{lemma}\label{enumeration_of_comps}
    Let $r,R\in\ZZ^+$. Then
    \begin{align*}
        \left|\mathscr{C}_{r,\text{\normalfont odd}}^{\leq R}\right|={\left\lfloor\frac{R-r}{2}\right\rfloor+r\choose r}.
    \end{align*}
\end{lemma}

%Let $\mathcal{C}_{r,\text{odd}}$ be the combinatorial class of such compositions and $\mathcal{I}_{\text{odd}}$ be the combinatorial class of the positive odd integers. Building one of these restricted compositions is equivalent to building a sequence of positive odd integers of length $r$, so $\mathcal{C}_{r,\text{odd}}$ admits the specification
 %   \begin{align*}
   %     \mathcal{C}_{r,\text{odd}}=\textsc{Seq}_{r}(\mathcal{I}_{\text{odd}}),
  %  \end{align*}
%    which by \cite[Theorem I.3]{analytic_combo} directly translates to the generating function

\begin{proof}
    We first compute the generating function $C(z)$ of the enumeration of compositions with exactly $r$ parts such that each is odd. A standard result on the combinatorics of compositions (see, e.g., \cite[Theorem 2.1]{GF_of_restricted_comps}) lends the formula
    \begin{align*}
        C(z)=\left(\sum_{n\geq0}z^{2n+1}\right)^r=\frac{z^r}{\left(1-z^2\right)^r}.
    \end{align*}
    Note that the $j$\textsuperscript{th} coefficient of $C(z)$ is nonzero only if $j-r$ is nonnegative and even. In particular, by a usual application of Newton's generalized binomial theorem,
    \begin{align}\label{newton_expansion}
        [z^{2n+r}]\frac{z^r}{\left(1-z^2\right)^r}={n+r-1\choose r-1}.
    \end{align}
    So by \eqref{newton_expansion} and the change of index $j\mapsto 2j+(r\text{~mod~} 2)$,
    \begin{align*}
        \left|\mathscr{C}_{r,\text{\normalfont odd}}^{\leq R}\right|&=\sum_{\substack{r\leq j\leq R\\ j-r\text{~even}}}{\frac{j-r}{2}+r-1\choose r-1}=\sum_{j=\lfloor r/2\rfloor}^{\left\lfloor\frac{R-(r\text{~mod~}2)}{2}\right\rfloor}{j-\lfloor r/2\rfloor+r-1\choose r-1}\\
        &=\sum_{j=0}^{\left\lfloor\frac{R-(r\text{~mod~}2)}{2}\right\rfloor-\left\lfloor r/2\right\rfloor}{j+r-1\choose r-1}\stackrel{*}{=}{\left\lfloor\frac{R-r}{2}\right\rfloor+r\choose r},
    \end{align*}
    where the starred equality follows from the hockey-stick identity \cite{hockey_stick} and $\left\lfloor{(R-(r\text{~mod~}2))}/2\right\rfloor-\left\lfloor r/2\right\rfloor=\lfloor (R-r)/2\rfloor$.
\end{proof}

The remainder of this section is dedicated to finding the maximum of the product in \eqref{eqn_comb_formula}. The critical step for maximizing this product is obtaining an upper bound for the difference ${x+m\choose k}-{x\choose k}$. Before proving this upper bound, we first note a simple bound by induction on the sum of powers of the first $a\in\ZZ^+$ positive integers.

\begin{lemma}\label{faulhabers_rule_bound}
    Let $a,k\in\ZZ^+$. Then $\sum_{i=1}^{a}i^{k-1}\leq a^k$.
\end{lemma}

\begin{proof}
    This result lends a simple proof by induction over $a\in\ZZ^+$. Noting that the base case of $a=1$ holds, the induction step follows from
    \begin{align*}
        \sum_{i=1}^{a+1}i^{k-1}&=(a+1)^{k-1}+\sum_{i=1}^{a}i^{k-1}\leq (a+1)^{k-1}+a^k\\
        &={k\choose k}a^k+\sum_{j=0}^{k-1}{k-1\choose j}a^j\leq \sum_{j=0}^{k}{k\choose j}a^j=(a+1)^k.
    \end{align*}
\end{proof}

With Lemma~\ref{faulhabers_rule_bound} noted, we prove our upper bound on a difference of binomial coefficients, which is the backbone of the proof of Theorem~\ref{upper_bound_thm}.

\begin{lemma}\label{difference_bound}
    Let $m,k\in\ZZ^+$. Then for any nonnegative integer $x\geq k-m$,
    \begin{align*}
        {x+m\choose k}-{x\choose k}\leq m(x+m+1-k)^{k-1}.
    \end{align*}
\end{lemma}

\begin{proof}
    We prove this claim by induction over $k\in\ZZ^+$. The base case of $k=1$ holds since
    \begin{align*}
        {x+m\choose 1}-{x\choose 1}=m\leq m(x+m+1-k)^0.
    \end{align*}
    For the induction step, employing Pascal's rule, we obtain
    \begin{align*}
        {x+m\choose k+1}-{x\choose k+1}&={x+m-1\choose k+1}-{x-1\choose k+1}+{x+m-1\choose k}-{x-1\choose k}\\
        &\leq {x+m-1\choose k+1}-{x-1\choose k+1}+m(x+m-k)^{k-1}\\
        &\leq {x+m-2\choose k+1}-{x-2\choose k+1}+\sum_{i=1}^{2}m(x+m+1-i-k)^{k-1}.
    \end{align*}
    Continuing in this way, for sufficiently small $j$, we have
    \begin{align*}
        {x+m\choose k+1}-{x\choose k+1}\leq {x+m-j\choose k+1}-{x-j\choose k+1}+\sum_{i=1}^{j}m(x+m+1-i-k)^{k-1}.
    \end{align*}
    In particular, setting $j=x+m-k$ yields   
    \begin{align*}    
        {x+m\choose k+1}-{x\choose k+1}\leq\sum_{i=1}^{x+m-k}m(x+m+1-i-k)^{k-1}=m\sum_{i=1}^{x+m-k}i^{k-1},
    \end{align*}
    which is at most $m(x+m-k)^{k}$ by Lemma~\ref{faulhabers_rule_bound}.
\end{proof}

With Lemma~\ref{difference_bound}, we can bound the individual factors in \eqref{eqn_comb_formula}, effectively reducing the proof of Theorem~\ref{upper_bound_thm} to maximizing the product of these upper bounds. 

We now recall the definition of the indicator function: given a set $X$, the indicator function of $A\subseteq X$ is the function
    \begin{align*}
        \mathds{1}_{A}(x)\defeq\begin{cases}
            1 & \text{if~}x\in A,\\
            0 & \text{if~}x\not\in A.
        \end{cases}
    \end{align*}
for all $x\in X$. For $r,R\in\ZZ^+$, we remark that $\mathds{1}_{2\ZZ+R}(r)$ is simply $1$ if $r,R$ are of the same parity, and $0$ otherwise. We use this expression multiple times in the proof of the following lemma.

%\textcolor{red}{define indicator function, remind that $\mathds{1}_{2\ZZ+R}(r)$ is simply $1$ if $r,R$ are of same parity and $0$ if not}

\begin{lemma}\label{maximizing_lemma}
    Given $r,R\in\ZZ^+$, the maximum of $\prod_{i=1}^{r}a_i$ over $\mathscr{C}_{r,\text{\normalfont odd}}^{\leq R}$ is obtained by $(\sigma_{r,R}(1),\dots,\sigma_{r,R}(r))$, where
    \begin{align*}
        \sigma_{r,R}(i)=
        \begin{cases}
            \oddleq\mleft(R/r\mright)+2 & \text{\normalfont if~}i\leq \frac{1}{2}{\evenleq\mleft(R-r\oddleq\mleft(R/r\mright)\mright),}\\
            \oddleq\mleft(R/r\mright) & \text{\normalfont otherwise.}
        \end{cases}
    \end{align*}
\end{lemma}

\begin{proof}
    In this proof, we call the tuple maximizing $\prod_{i=1}^{r}a_i$ over $(a_1,\dots,a_r)\in\mathscr{C}_{r,\text{odd}}^{\leq R}$ the\footnote{This tuple is indeed unique up to a permutation of its elements.} \emph{maximizing tuple}. If the maximizing tuple $(a_1,\dots,a_r)$ is such that $a_{j_1}>a_{j_2}+2$ for some $j_1,j_2$, then $(a_{j_1}-2)(a_{j_2}+2)>a_{j_{1}}a_{j_2}$, giving that
    \begin{align*}
        \prod_{i=1}^{r}a_i<(a_{j_1}-2)(a_{j_2}+2)\prod_{\substack{1\leq i\leq r\\ i\neq j_1,j_2}}a_i.
    \end{align*}
    So the maximizing tuple $(a_1,\dots,a_r)$ is such that $|a_{h_1}-a_{h_2}|\leq 2$ for all $h_1,h_2$. In particular, this means that the maximizing tuple is of the form $(t,\dots,t,t+2,\dots,t+2)$ up to a permutation of its elements, where the number of $(t+2)$'s is possibly zero.

    Recall that the elements of any tuple in $\mathscr{C}_{r,\text{odd}}^{\leq R}$ can sum to at most $R-1+\mathds{1}_{2\ZZ+R}(r)$ and that $r\oddleq(R/r)\leq R$, with equality only if $R/r\in2\ZZ+1$, which implies that $r\equiv R\text{~mod~} 2$. So, as the sum of the elements of the maximizing tuple is maximal, we can always take $t\geq \oddleq(R/r)$. However, if $t\geq \oddleq(R/r)+2$, then the sum of the elements of the tuple is at least
    \begin{align}\label{sum_of_elements_contra}
        \sum_{i=1}^{r}(\oddleq(R/r)+2)=r(\oddleq(R/r)+2)\geq r(\lfloor R/r\rfloor+1)\geq R+1,
    \end{align}
    a contradiction. So $t=\oddleq(R/r)$. Consequently, the problem reduces to maximizing the number of $(\oddleq(R/r)+2)$'s, which we henceforth denote as $H$, in the maximizing tuple.

    Noting that the sum of the elements of the maximizing tuple is $r\oddleq(R/r)+2H=(r-H)\oddleq(R/r)+H(\oddleq(R/r)+2)=R-1+\mathds{1}_{2\ZZ+R}(r)$, we find that the maximum value of $H$ is half of the largest even integer less than or equal to $R-r\oddleq(R/r)$. So the maximizing tuple is $(\sigma_{r,R}(1),\dots,\sigma_{r,R}(r))$, up to a permutation of its elements.
\end{proof}

To prove Theorem~\ref{upper_bound_thm}, starting from the formula given in Proposition~\ref{combinatorial_formula_for_number_of_occurrences}, we bound each factor in the product using Lemma~\ref{difference_bound} and maximize the product of these bounds using Lemma~\ref{maximizing_lemma}. Multiplying this maximum by the size of $\mathscr{C}_{r,\text{odd}}^{\leq R'}$ computed in Lemma~\ref{enumeration_of_comps} then yields the upper bound as desired.

\begin{proof}[\textbf{Proof of Theorem~\ref{upper_bound_thm}}]
    We first bound the product in \eqref{eqn_comb_formula}. Writing $p=x_1^{u_1}\cdots x_r^{u_r}$ and $w=y_1^{v_1}\cdots y_{R}^{v_{R}}$, let 
    \begin{align*}
        \mathcal{S}_{p,w}\defeq\left\{C\in \mathscr{C}_{r,\text{odd}}^{\leq R'}:\sum_{j\in A_i}v_{j+1-\delta_{x_1,y_1}}\geq u_i \text{~for all~} A_i\in C([R'])\right\}.
    \end{align*}
    That is, $\mathcal{S}_{p,w}$ is the set of tuples in $\mathscr{C}_{r,\text{odd}}^{\leq R'}$ whose corresponding product in \eqref{eqn_comb_formula} is nonzero. Let $(c_1,\dots,c_r)=C\in\mathcal{S}_{p,w}$ and $A_i\in C([R'])$ and $v_j'=v_{j+1-\delta_{x_1,y_1}}$. As $\sum_{j\in A_i}v'_{j}-u_i\geq 0$, by Lemma~\ref{difference_bound},
    \begin{align}\label{usual_estimate}
        {\sum_{j\in A_i}v'_{j}\choose u_i}-{\sum_{j\in A_i\setminus\{\max(A_i)\}}v'_{j}\choose u_i}\leq v'_{\max(A_i)}\left(\sum_{j\in A_i}v'_{j}-u_i+1\right)^{u_i-1}.
    \end{align}
    Noting that $\sum_{j\in A_i}v'_{j}\leq |A_i|\cdot R_{\max}'= R_{\max}'(c_i+1)/2$, we have 
    \begin{align}\label{RHS_maxing}
        v'_{\max(A_i)}\left(\sum_{j\in A_i}v'_{j}-u_i+1\right)^{u_i-1}\leq R_{\max}'\left(\frac{R_{\max}'}{2}(c_i+1)-r_{\min}+1\right)^{r_{\max}-1}.
    \end{align}
    We now note that any $r$-tuple in $\mathscr{C}_{r,\text{odd}}^{\leq R'}$ that maximizes $\prod_{i=1}^{r}a_i$ over all $(a_1,\dots,a_r)\in\mathscr{C}_{r,\text{odd}}^{\leq R'}$ also maximizes
    \begin{align*}
        \prod_{i=1}^{r}R_{\max}'\left(\frac{R_{\max}'}{2}(a_i+1)-r_{\min}+1\right)^{r_{\max}-1},
    \end{align*}
    provided that each $(R_{\max}'(a_i+1)/2)-r_{\min}\geq0$. By Lemma~\ref{maximizing_lemma}, an $r$-tuple maximizing $\prod_{i=1}^{r}a_i$ over $\mathscr{C}_{r,\text{odd}}^{\leq R'}$ is $(\sigma_{r,R'}(1),\dots,\sigma_{r,R'}(r))$. Note that at least one $c_h$ must be at most $\oddleq(R'/r)$; otherwise, $\sum_{i=1}^{r}c_i\geq R'+1$, following from the same computation as in \eqref{sum_of_elements_contra}. 
    
    Now verifying that each $(R_{\max}'(\sigma_{r,R'}(i)+1)/2)-r_{\min}\geq0$, by assumption on $C$, we have for all $1\leq i\leq r$ that
    \begin{align*}
        \frac{R_{\max}'}{2}(\sigma_{r,R'}(i)+1)-r_{\min}
        \geq\frac{R_{\max}'}{2}(c_h+1)-r_{\min} \geq \sum_{j\in A_h}v'_{j}-u_h\geq 0.
    \end{align*}
    Thus, the product of the right-hand side of \eqref{RHS_maxing} over all $i$ is at most
    \begin{align*}
        \prod_{i=1}^{r}R_{\max}'\left(\frac{R_{\max}'}{2}(\sigma_{r,R'}(i)+1)-r_{\min}+1\right)^{r_{\max}-1}.
    \end{align*}
    Note that the product in \eqref{eqn_comb_formula} corresponding to any $C\in\mathscr{C}_{r,\text{odd}}^{\leq R'}\setminus\mathcal{S}_{p,w}$ is zero. Therefore, by Proposition~\ref{combinatorial_formula_for_number_of_occurrences} and Lemma~\ref{enumeration_of_comps},
    \begin{align*}
        c_p(w)=c_p(w')&=\sum_{C\in\mathcal{S}_{p,w}}\prod_{A_i\in C([R'])}\left({\sum_{j\in A_i}v'_{j}\choose u_i}-{\sum_{j\in A_i\setminus\{\max(A_i)\}}v'_{j}\choose u_i}\right)\\
        &\leq \sum_{C\in\mathcal{S}_{p,w}}\prod_{i=1}^{r}R_{\max}'\left(\frac{R_{\max}'}{2}(\sigma_{r,R'}(i)+1)-r_{\min}+1\right)^{r_{\max}-1}\\
        &\leq\left|\mathscr{C}_{r,\text{\normalfont odd}}^{\leq R'}\right|\prod_{i=1}^{r}R_{\max}'\left(\frac{R_{\max}'}{2}(\sigma_{r,R'}(i)+1)-r_{\min}+1\right)^{r_{\max}-1}\\
        &={\left\lfloor\frac{R'-r}{2}\right\rfloor+r\choose r}\prod_{i=1}^{r}R_{\max}'\left(\frac{R_{\max}'}{2}(\sigma_{r,R'}(i)+1)-r_{\min}+1\right)^{r_{\max}-1}.
    \end{align*}
    Lastly, letting $p$ and $w$ be alternating, by observing the chain of equalities
    \begin{align*}
        c_p(w)&=c_p(w')=\sum_{C\in\mathscr{C}_{r,\text{odd}}^{\leq R'}}\prod_{A_i\in C([R'])}\left({\sum_{j\in A_i}1\choose 1}-{\sum_{j\in A_i\setminus\{\max(A_i)\}}1\choose 1}\right)\\
        &={\left\lfloor\frac{R'-r}{2}\right\rfloor+r\choose r}={\left\lfloor\frac{R'-r}{2}\right\rfloor+r\choose r}\prod_{i=1}^{r}1\left(\frac{1}{2}(\sigma_{r,R'}(i)+1)-1+1\right)^{1-1},
    \end{align*}
    we see that equality is indeed obtained as claimed.
\end{proof}

\begin{proof}[Proof of Remark~\ref{relaxed_upper_bound}]
    The problem of maximizing $\prod_{i=1}^{r}a_i$ over all $(a_1,\dots,a_r)\in\RR^r_{+}$ given the constraint $\sum_{i=1}^{r}a_i\leq R'$ follows from the AM-GM inequality. In particular, as seen in \cite[Lemma 1]{AM_GM_max}, the $r$-tuple maximizing $\prod_{i=1}^{r}a_i$ over $\RR^r_{+}$ given that $\sum_{i=1}^{r}a_i= R'$ is $(R'/r,\dots,R'/r)$. Since $R'/r\geq\oddleq (R'/r)$, we have $R_{\max}((R'/r)+1)/2-r_{\min}\geq0$. Thus, the product of the right-hand side of \eqref{RHS_maxing} over all $i$ is at most
    \begin{align*}
        \left(R_{\max}'\left(\frac{R_{\max}'}{2}\left(\frac{R'}{r}+1\right)-r_{\min}+1\right)^{r_{\max}-1}\right)^r,
    \end{align*}
    from which Remark~\ref{relaxed_upper_bound} follows.
\end{proof}

\section{Directions for Future Research}\label{section_4}

\subsection{Generalizing Theorem~\ref{polynomial_theorem} to $m$-ary words}\label{polynomial_m_ary_section}

Let $B_p^{(m)}(n,k)$ be the number of $m$-ary words of length $n$ with exactly $k$ occurrences of $p$. Before generalizing Theorem~\ref{polynomial_theorem} to $m$-ary words, we first need to determine what form we expect $B_p^{(m)}(n,k)$ to take. In \cite{subseq_freq}, it was shown that $\deg(B_p(n,1))=\ell-r+1$ by a balls-and-bins interpretation. In this paper, a complicated and delicate argument was given in the proof of Theorem~\ref{polynomial_theorem} to show that we could apply a morally similar balls-and-bins interpretation to $B_p(n,k)$ for all $k\in\ZZ^+$. A natural question following from this observation in the binary case is: do we expect $B_p^{(m)}(n,k)$ to take a form similar to $B_p^{(m)}(n,1)$ for any positive integer $k$? We expect the answer to be yes, and perhaps expectedly, by a generalized balls-and-bins interpretation. To motivate our forthcoming conjecture on the form of $B_p^{(m)}(n,k)$, we generalize Proposition~\ref{menon_singh_result_k_1} in the following proposition, the proof of which we defer to Appendix~\ref{appendix}.

\begin{prop}\label{m_ary_k_1}
    Let $p$ be an $m$-ary word of length $\ell$ with $r$ runs. Then $B_p^{(m)}(n,1)=q_1(n)(m-1)^n+q_2(n)(m-2)^n$, where $q_1$ is a polynomial of degree $\ell-r+1$ and $q_2$ is a polynomial of degree $r-2$. If $r=1$, take $q_2=0$.
\end{prop}

The polynomials $q_1,q_2$ are dependent only on $\ell$, $r$, and $m$; however, for our purposes of substantiating the following conjecture, merely knowing their degrees is sufficient.

\begin{conj}\label{m_ary_polynomial_conj}
    Let $k\in\ZZ^+$ and $p$ be an $m$-ary word of length $\ell$ with $r\geq2$ runs. Then for $n$ sufficiently large, 
    \begin{align*}
        B^{(m)}_p(n,k)=\sum_{i=1}^{m-1}q_i(n)(m-i)^n,
    \end{align*}
    where each $q_i$ is a polynomial and $\deg(q_1)\leq\ell-r+1$. Furthermore, if either
    \begin{enumerate}[(i)]
        \item $p$ has a run of length $1$, or
        \item $p$ has a run of length $b\geq2$ at its boundary such that ${b+s\choose b}=k$ for some $s\geq0$,
    \end{enumerate}
    then $\deg(q_1)=\ell-r+1$.
\end{conj}

We note that, upon seeing a preliminary draft of this paper, Valentino Vito \cite{vito_communication} was able to resolve Conjecture~\ref{m_ary_polynomial_conj} in the affirmative using ChatGPT 5.5. An article regarding this resolution is in preparation \cite{vito_paper}.

Observe that the $m=2$ case of Conjecture~\ref{m_ary_polynomial_conj} is simply Theorem~\ref{polynomial_theorem}. The second part of Conjecture~\ref{m_ary_polynomial_conj} follows from a straightforward generalization of the constructions given in the proof of Theorem~\ref{polynomial_theorem} via Proposition~\ref{m_ary_k_1}. We expect that the proof of Theorem~\ref{polynomial_theorem} should be able to be generalized to prove Conjecture~\ref{m_ary_polynomial_conj}; however, there are further considerations in the generalized case which must be taken into account.

\begin{definition}[Naive Generalization of $\text{SW}_p(k)$]
    Let $p$ be an $m$-ary word. The \textbf{naive generalization} of $\text{\normalfont SW}_p(k)$ is the set of words on $\{0,1,\dots,m-1,\hatted{0},\hatted{1},\dots,\hatted{m-1}\}$ with exactly $k$ occurrences of $p$ such that
    \begin{enumerate}[(i)]
        \item all runs of non-hatted letters are short,
        \item all runs of $\hatted{a}$'s are of length $1$ and do not have any $a$'s as surrounding letters for all $0\leq a\leq m-1$, and
        \item for $x\in\{0,1,\dots,m-1\}$, replacing $\hatted{x}$ with a run of $x$'s of arbitrary length does not cause the number of occurrences of $p$ to exceed $k$.
    \end{enumerate}
\end{definition}

We observe that the naive generalization of $\text{SW}_p(k)$ quickly leads to issues. In particular, this naive generalization of $\text{SW}_p(k)$ induces this set to no longer be finite for $m\geq3$. The generalized proof of Lemma~\ref{short_words_constant_in_n} fails since for $m\geq3$, an $m$-ary word can have arbitrarily many runs without a single occurrence of another $m$-ary word $p$ (e.g., $p=012$ and $w=010101\cdots$). Thus, a successful generalization of the proof of Theorem~\ref{polynomial_theorem} requires a refinement of Definition~\ref{SW_def}, which we expect to correspond to the generalized proof harboring further difficulties and complications. Nevertheless, if the right refinement of Definition~\ref{SW_def} is constructed, Conjecture~\ref{m_ary_polynomial_conj} may very well follow by much of the same reasoning as in the proof of Theorem~\ref{polynomial_theorem}.

\subsection{Completely determining when $\deg(B_p(n,k))=\ell-r+1$}\label{second_part_section}

In Theorem~\ref{polynomial_theorem}, assuming that $r\geq2$, we give two conditions on $p$ and $k$ sufficient for the degree of $B_p(n,k)$ to match its upper bound of $\ell-r+1$. By \cite[Propositions 3.2, 3.4, 3.8, 3.11]{subseq_freq}, we observe that for $1\leq k\leq4$, at least one of these conditions being satisfied is also necessary for $\deg(B_p(n,k))=\ell-r+1$. Originally, this made us conjecture that this holds for all $k\in\ZZ^+$. However, upon seeing a preliminary draft of this paper, Valentino Vito \cite{vito_communication} produced the counterexample of $k=5$ and $p=0011$ using ChatGPT 5.5. Thus, since these two conditions are not exhaustive as we initially thought, we propose completely determining when $\deg(B_p(n,k))=\ell-r+1$.

%\begin{conj}\label{conj_conditions_are_also_necessary}
   % Let $k\in\ZZ^+$ and $p$ be a binary word of length $\ell$ with $r\geq2$ runs. Then for $n$ sufficiently large, $\deg(B_p(n,k))=\ell-r+1$ if and only if either
   % \begin{enumerate}[(i)]
    %    \item $p$ has a run of length $1$, or
     %   \item $p$ has a run of length $b\geq2$ at its boundary such that ${b+s\choose b}=k$ for some $s\geq0$.
    %\end{enumerate}
%\end{conj}

\subsection{Generalizing Theorem~\ref{upper_bound_thm} to $m$-ary words}\label{m_ary_bound}

To generalize Theorem~\ref{upper_bound_thm} to $m$-ary words, one must first generalize Proposition~\ref{combinatorial_formula_for_number_of_occurrences}. We expect such a combinatorial formula to be of a similar form; however, the indexing set in the summation will not be as nice as in the binary case since the runs of $w$ need no longer be periodic in value. Correspondingly, the $m$-ary analog of each $A_i$ will depend on more parameters and be less simple in general. Once an $m$-ary combinatorial formula for $c_p(w)$ is proven, we expect a strategy similar to that of the proof of Theorem~\ref{upper_bound_thm} to yield the desired bound for $m$-ary words.

\subsection{Improving Theorem~\ref{upper_bound_thm}}\label{improving_upper_bound}

While our inequality is sharp, the process of bounding the product in \eqref{eqn_comb_formula} in particular lends many paths toward improvement. In particular, the elementary nature of Lemma~\ref{difference_bound} suggests that improvement is well within reach. Since ${x+m\choose k}-{x\choose k}=\frac{m}{(k-1)!}x^{k-1}+O(x^{k-2})$, using a monomial bound as in Lemma~\ref{difference_bound} is natural, but using other polynomials of degree $k-1$ or even other functions which are $\Theta(x^{k-1})$ is likely to be fruitful.

Another improvement lies in computing the size of $\mathcal{S}_{p,w}$, or at least obtaining better upper bound than $|\mathcal{S}_{p,w}|\leq |\mathscr{C}_{r,\text{odd}}^{\leq R'}|$. A nonzero lower bound for $|\mathcal{S}_{p,w}|$ would also be interesting, as it would lead to a nonzero lower bound for $c_p(w)$. We expect obtaining these nontrivial bounds to be very difficult for general $p$ and $w$.

One could also restrict to particular classes of $p$ and $w$ to yield a better upper bound on $c_p(w)$ over these classes. Working in complete generality induces Theorem~\ref{upper_bound_thm} to be crude to account for edge cases (e.g., when all runs of $p$ and $w$ are of the same length, respectively). Naturally, making more assumptions on the run structures of $p$ and $w$ will permit various refinements of Theorem~\ref{upper_bound_thm}.

\appendix

\section{Supplementary Proofs}\label{appendix}

\subsection{Proof of Example~\ref{upper_bound_asymptotics}}
    If $p$ is alternating, then $r_{\max}=1$ and $r=\ell$, inducing the simplification $B(n)=\binom{\lfloor(R-\ell)/2\rfloor+\ell}{\ell}R_{\max}^{\ell}$. So
    \begin{align*}
        B(n)&=\binom{\lfloor(R-\ell)/2\rfloor+\ell}{\ell}R_{\max}^{\ell}=\binom{\lfloor(a\sqrt{n}(1+o(1))-\ell)/2\rfloor+\ell}{\ell}\left(b\sqrt{n}(1+o(1))\right)^{\ell}\\
        &={\frac{a}{2}\sqrt{n}(1+o(1))\choose\ell}\left(b\sqrt{n}(1+o(1))\right)^{\ell}=\left(\frac{(a/2)^{\ell}}{\ell!}n^{\ell/2}(1+o(1))\right)\left(b^{\ell}n^{\ell/2}(1+o(1))\right)\\
        &=\frac{(ab/2)^{\ell}}{\ell!}n^{\ell}(1+o(1)).
    \end{align*}
    Noting that $T(n)={n\choose\ell}=\frac{1}{\ell!}n^{\ell}(1+o(1))$, it follows that $T(n)/B(n)$ tends to $(ab/2)^{-\ell}$.

\subsection{Proof of Proposition~\ref{m_ary_k_1}}
    This proof generalizes \cite[Proposition 3.2]{subseq_freq} via a generalized balls-and-bins interpretation. In particular, like the proof of \cite[Proposition 3.2]{subseq_freq}, we determine which letters can be inserted into the spaces between letters of $p$ without inducing the existence of another occurrence of $p$.

    Consider a space between two distinct letters of $p$. Then we can insert letters not equal to these two letters, of which there are $m-2$, in this space. Noting that such spaces are exactly the spaces between two consecutive runs of $p$, it follows that there are $r-1$ of these spaces.

    Consider a space between two equal letters of $p$. Then we can insert letters not equal to the value of these two equal letters, of which there are $m-1$, in this space. Noting that such a space is in the middle of a run in $p$, it follows that the number of these spaces is $\ell-r$.

    Lastly, consider the spaces before the first letter of $p$ and after the last letter of $p$. Before the first letter of $p$, we can insert letters not equal to $p_1$, of which there are $m-1$, in this space. After the last letter of $p$, we can insert letters not equal to $p_{\ell}$, of which there are $m-1$, in this space.

    Thus, there are $\ell-r+2$ spaces of $p$ where $m-1$ letters can be inserted, and $r-1$ spaces of $p$ where $m-2$ letters can be inserted. For each $1\leq i\leq \ell-r+2$, let $a_i$ be the number of letters we insert into the $i$\textsuperscript{th} space of $p$ not between two consecutive runs. For each $\ell-r+3\leq i\leq \ell+1$, let $a_i$ be the number of letters we insert into the $i$\textsuperscript{th} space of $p$ between two consecutive runs. Then
    \begin{align*}
        B_p^{(m)}(n,1)=\sum_{a_1+\cdots+a_{\ell+1}=n-\ell}(m-1)^{\sum_{i=1}^{\ell-r+2}a_i}(m-2)^{\sum_{i=\ell-r+3}^{\ell+1}a_i}.
    \end{align*}
    Observing that the above summation is an $(\ell+1)$-fold Cauchy product, it follows that
    \begin{align*}
        \sum_{n\geq0} B_p^{(m)}(n,1) z^n=\frac{z^{\ell}}{(1-(m-1)z)^{\ell-r+2}(1-(m-2)z)^{r-1}}.
    \end{align*}
    By \cite[Theorem IV.9]{analytic_combo}, we have that $B_p^{(m)}(n,1)=q_1(n)(m-1)^n+q_2(n)(m-2)^n$, where $q_1$ is a polynomial of degree $\ell-r+1$ and $q_2$ is a polynomial of degree $r-2$ (if $r=1$, take $q_2=0$), as desired.

%\footnote{The Cauchy product is the discrete convolution of the coefficients of two power series.}

\section*{Acknowledgments}
This research was conducted at the 2025 Duluth REU with support from Jane Street Capital, NSF Grant 2409861, and donations from Ray Sidney and Eric Wepsic. We thank Joseph Gallian and Colin Defant for organizing the REU and providing this incredible opportunity. We recognize Noah Kravitz for his indispensable mentorship throughout the program and for suggesting this project. The writing of this paper was greatly improved under the suggestions of Eliot Hodges, Sean Li, Alex Moon, Jacob Paltrowitz, Carl Schildkraut, Alec Sun, and Sophie Zhu. We are grateful to Ruben Carpenter and Jacob Paltrowitz for identifying an error in a previous version of this work. We also thank Mitchell Lee, Rupert Li, and Maya Sankar for their advice throughout the program.

\bibliography{main}{}
\bibliographystyle{plain}

%Note that only runs of length at most $k$ are used to produce each of the $k$ occurrences of $p$ in $w$; otherwise, $c_p(w)>k$.

%, note that there is only one slot preceding $w_{h_1}$ to insert a long run of $(1-y_1)$'s since $h_1<\cdots<h_{\ell}$ is the first occurrence of $p$ in $w$, and there is only one slot following $w_{h_{\ell}'}$ to insert a long run of $(1-y_2)$'s since $h_1'<\cdots<h_{\ell}'$ is the last occurrence of $p$ in $w$.
    
    %Suppose that $p$ begins with $y_1$ and ends with $y_2$. Suppose a slot is not between a $u,u'$ pair. Then this slot is not in the middle of an occurrence of $p$ in $w$. Thus, this slot is either adjacent to $(1-y_1)$'s in the first run of $w$, or adjacent to $(1-y_2)$'s in the last run of $w$. Let $h_1<\cdots<h_{\ell}$ be the first occurrence of $p$ in $w$ and $h_1'<\cdots<h_{\ell}'$ be the last occurrence of $p$ in $w$. Then inserting a long run of $y_1$'s before $w_{h_1}$ or a long run of $y_2$'s after $w_{h_{\ell}'}$ both induce $c_p(w)>k$. So only long runs of $(1-y_1)$'s before $w_{h_1}$ and long runs of $(1-y_2)$'s after $w_{h_{\ell}'}$ are permitted. Finally, up to equality of the word, note that there is only one slot preceding $w_{h_1}$ to insert a long run of $(1-y_1)$'s since $h_1<\cdots<h_{\ell}$ is the first occurrence of $p$ in $w$, and there is only one slot following $w_{h_{\ell}'}$ to insert a long run of $(1-y_2)$'s since $h_1'<\cdots<h_{\ell}'$ is the last occurrence of $p$ in $w$.

\end{document}